\documentclass[english,letterpaper,11pt,reqno]{amsart}
\usepackage[backend=bibtex,style=alphabetic,maxnames=6]{biblatex}
\usepackage{appendix}
\usepackage{amsbsy}
\usepackage{scalerel}
\usepackage{tcolorbox}
\usepackage{amsfonts}
\usepackage{amsmath}
\usepackage{amssymb}
\usepackage{amsthm}
\usepackage{graphicx}
\usepackage{ifthen}
\usepackage{textcomp}
\usepackage{enumitem, color, amssymb}

\usepackage{dsfont}
\usepackage{mathrsfs}
\usepackage{calrsfs}
\usepackage[bookmarksnumbered,colorlinks]{hyperref}
\hypersetup{colorlinks=true, linkcolor=blue, citecolor=blue}
\usepackage{cancel}
\usepackage{setspace}

\newcommand{\lra}{\longrightarrow}

\newcommand{\RR}{\mathbb{R}}

\newcommand{\vep}{\varepsilon}

\makeatletter
\newcommand*{\defeq}{\mathrel{\rlap{%
                     \raisebox{0.25ex}{$\m@th\cdot$}}%
                     \raisebox{-0.25ex}{$\m@th\cdot$}}%
                     =}
\makeatother

\makeatletter
\newcommand*\owedge{\mathpalette\@owedge\relax}
\newcommand*\@owedge[1]{%
  \mathbin{%
    \ooalign{%
      $#1\m@th\bigcirc$\cr
      \hidewidth$#1\m@th\wedge$\hidewidth\cr
    }%
  }%
}
\makeatother

\newtheorem{thm}{Theorem}
\newtheorem{lemma}{Lemma}
\newtheorem{cor}{Corollary}

\newtheorem{prop}{Proposition}
\newtheorem*{definition-non}{Definition}
\newtheorem*{theorem-non}{Theorem}
\newtheorem*{proposition-non}{Proposition}
\newtheorem*{lemma-non}{Lemma}
\newtheorem*{corollary-non}{Corollary}

\newcommand{\beqa}{\begin{eqnarray}}
\newcommand{\beq}{\begin{equation}}
\newcommand{\eeqa}{\end{eqnarray}}
\newcommand{\eeq}{\end{equation}}

\newcommand\ipr[2]{\langle {#1},{#2}\rangle_{\!\scalebox{0.6}{$g$}}}
\newcommand\iph[2]{\langle {#1},{#2}\rangle_{\!\scalebox{0.6}{$h$}}}

\newcommand\ww[2]{#1 \wedge #2}

\newcommand\imp{\hspace{.2in}\Rightarrow\hspace{.2in}}

\newcommand\cd[2]{\nabla_{\!#1}{#2}}

\newcommand\Ric{\text{Ric}_{\scalebox{0.6}{$g$}}}

\newcommand\Rmr{\text{Rm}_{\scalebox{0.6}{$g$}}}

\newcommand\comma{\hspace{.2in},\hspace{.2in}}
\newcommand\commas{\hspace{.1in},\hspace{.1in}}
\newcommand\commass{\hspace{.05in},\hspace{.05in}}

\newcommand\hsh{*_{\scalebox{0.5}{$h$}}}

\newcommand\hsr{*_{\scalebox{0.6}{$g$}}}

\newcommand\gsec{\text{sec}_{\scalebox{0.6}{$g$}}}
\newcommand\co{\hat{R}_{\scalebox{0.6}{$g$}}}
\newcommand\cow{\hat{R}_{\scalebox{0.6}{\scalebox{0.7}{$\wedge$}}}}
\newcommand\Ls[2]{\Lambda^{#2}_{\scalebox{0.6}{$#1$}}}

\newcommand\coh{\hat{R}_{\scalebox{0.6}{$g$-$h$}}}

\newcommand\Rsec{\sigma_{\scalebox{0.6}{$g$}}}
\newcommand\ghsec{\sigma_{\scalebox{0.6}{$g$-$h$}}}

\newcommand\sgh{\text{scal}_{\scalebox{0.6}{$g$-$h$}}}
\newcommand\Ricgh{\text{Ric}_{\scalebox{0.6}{$g$-$h$}}}
\newtheorem{innercustomgeneric}{\customgenericname} 
\providecommand{\customgenericname}{}
\newcommand{\newcustomtheorem}[2]{%
  \newenvironment{#1}[1]
  {%
   \renewcommand\customgenericname{#2}%
   \renewcommand\theinnercustomgeneric{##1}%
   \innercustomgeneric
  }
  {\endinnercustomgeneric}
}

\makeatletter
\def\Ddots{\mathinner{\mkern1mu\raise\p@
\vbox{\kern7\p@\hbox{.}}\mkern2mu
\raise4\p@\hbox{.}\mkern2mu\raise7\p@\hbox{.}\mkern1mu}}
\makeatother

\newcommand{\Sph}{\mathbb{S}}
\newcommand{\CP}{\mathbb{CP}}

\newcommand{\dVh}{dV_{\scalebox{0.6}{$h$}}}
\newcommand{\dVg}{dV_{\scalebox{0.6}{$g$}}}
\newcommand\qh{q_{\scalebox{0.6}{$h$}}}
\newcommand\rhoz{\rho_{\scalebox{0.6}{$0$}}}

\newcustomtheorem{customthm}{Theorem}
\newcustomtheorem{customlemma}{Lemma}

\begin{document}
\title[]{Hodge splittings and Einstein 4-manifolds}
\author[]{Amir Babak Aazami}
\address{Clark University\hfill\break\indent
Worcester, MA 01610}
\email{aaazami@clarku.edu}

\maketitle

\begin{abstract}
On an oriented $4$-manifold, we study pairs of Riemannian metrics $(g,h)$ for which the curvature tensor of $g$ preserves the Hodge splitting determined by $h$. This extends the Einstein condition in dimension four, which is recovered when $h$ is conformal to $g$. We prove that such pairs satisfy a generalized Hitchin--Thorpe inequality, which reduces to the classical one when $h$ is conformal to $g$.  We then exhibit a pair $(g,h)$ on $\#_5\CP^2$, which violates Hitchin--Thorpe and hence admits no Einstein metric, thus showing that our condition is indeed broader than the Einstein condition. \end{abstract}

\section{Introduction}
On an oriented $4$-manifold $M$, which pairs of Riemannian metrics have the property that the curvature tensor of one leaves the Hodge splitting of the other invariant? In other words, which pairs $(g,h)$ satisfy
\beqa
\label{eqn:gh}
\Rmr(\Ls{h}{+},\Ls{h}{-})=0,
\eeqa
where $\Rmr$ is the curvature bilinear form of $g$ and $\Lambda^2=\Ls{h}{+}\oplus\Ls{h}{-}$ is the decomposition of $\Lambda^2$ into the self-dual and anti-self-dual eigenspaces of $h$'s Hodge star $\hsh$? Because the Hodge star on $2$-forms is conformally invariant in dimension four, what is relevant to \eqref{eqn:gh} is actually the oriented conformal class $[h]$, or equivalently the splitting $\Lambda^2=\Ls{h}{+}\oplus\Ls{h}{-}$ itself. The guiding example is the Einstein case: some, and hence every, $h\in[g]$ satisfies \eqref{eqn:gh} if and only if $g$ is Einstein \cite{berger,thorpe2}.
\vskip 6pt
In fact \eqref{eqn:gh} separates the algebra of the curvature tensor from the analytic problem of varying $g$ itself. Let $Q$ denote the wedge pairing on $\Lambda^2$ defined by $Q(\xi,\eta)\dVg\defeq \xi\wedge\eta$, and let $\cow\colon\Lambda^2\lra\Lambda^2$ be the $Q$-self-adjoint endomorphism determined by $Q(\cow(\xi),\eta)\defeq \Rmr(\xi,\eta)$. Then in fact $\cow=\hsr\co$; if $g$ is Einstein, then $\hsr\co = \co\hsr$, so $\cow$ is real-diagonalizable.  But the latter is not true in general, as $Q$ has signature $(3,3)$. Our first result identifies this as the precise obstruction to \eqref{eqn:gh}, and then uses it to derive a topological obstruction when $M$ is compact.  Indeed, setting $\qh(\xi,\eta)\defeq Q(\xi,\hsh\eta)$ and
$$
\rhoz(g,h)
\defeq
\sup_M\left\|\frac12\left(\hsr+\hsh\!\hsr\!\hsh\right)\right\|_{\text{$\qh$}},
$$
we show that this quantity gauges the deviation from the classical Hitchin--Thorpe inequality \cite{thorpe3,hitchin}.

\begin{theorem-non}[Relative Hitchin--Thorpe inequality]
If a compact oriented 4-manifold $M$ admits a Riemannian pair $(g,h)$ satisfying \emph{$\Rmr(\Ls{h}{+},\Ls{h}{-})=0$}, then
$$
2\chi(M)\ge \frac{3}{\rhoz(g,h)}|\tau(M)|,
$$
with $\rhoz \geq 1$ and equal to $1$ if and only if $h\in[g]$\emph{;} the latter yields the classical Hitchin--Thorpe inequality. Finally, $\chi(M) = 0$ if and only if $g$ is flat.
\end{theorem-non}

Our second and main result is that $\rhoz$ is essential, because \eqref{eqn:gh} can violate Hitchin--Thorpe. In violating Hitchin--Thorpe, our example goes in a different direction from those obtained via Seiberg--Witten theory in \cite{lebrun3}.

\begin{theorem-non}[A Hitchin--Thorpe-violating example]
$\#_5\CP^2$ admits a Riemannian pair $(g,h)$ satisfying \emph{$\Rmr(\Ls{h}{+},\Ls{h}{-})=0$}, but no Einstein metric.
\end{theorem-non}

(Here $\rhoz(g,h) \geq \frac{15}{14}$.) Another aspect of \eqref{eqn:gh} that we examine is its variational character.  For fixed $g$, consider the $h$-scalar contraction of $g$'s curvature tensor,
$$
\Ricgh \defeq \mathrm{tr}_{\scalebox{0.6}{$h$}}\Rmr \comma \sgh\defeq \mathrm{tr}_{\scalebox{0.6}{$h$}}(\Ricgh).
$$
Then we show that on compact $M$, the functional
$$
\mathcal{A}_{\scalebox{0.6}{$g$}}(h)\defeq \int_M \sgh\,dV_{\scalebox{0.6}{$h$}},
$$
which is conformally invariant in $h$, has Euler--Lagrange equation 
\beqa
\label{eqn:Ricc_gen}
\Ricgh =\frac{\sgh}{4}h,
\eeqa
and that this is equivalent to \eqref{eqn:gh}.  In contrast to the usual Einstein-Hilbert action, which yields a second-order elliptic PDE, the variation of $\mathcal{A}_{\scalebox{0.6}{$g$}}$ here is algebraic: $g$ is fixed, no derivatives of the variation appear, and the equation is zero-order in $h$. Nevertheless, we present this as our third and final result because \eqref{eqn:Ricc_gen} generalizes the Einstein equation $\Ric = \lambda g$.
\vskip 6pt
The paper is organized as follows.  Section \ref{sec:quad_gen} reformulates the condition $\Rmr(\Ls{h}{+},\Ls{h}{-})=0$ in several equivalent ways, and also includes the variational characterization via $\mathcal{A}_{\scalebox{0.6}{$g$}}$.  Section \ref{sec:four} proves the diagonalizability criterion of $\cow$, the Euler obstruction, and the relative Hitchin--Thorpe inequality.  Section \ref{sec:examples} discusses examples and shared-frame curvature formulae.  Finally, Section \ref{sec:cp2} constructs the Riemannian pair $(g,h)$ on $\#_5\CP^2$ satisfying \eqref{eqn:gh}.

\section{Characterizing $\Rmr(\Ls{h}{+},\Ls{h}{-})=0$}
\label{sec:quad_gen}
We begin by reviewing the curvature operator (see, e.g., \cite{besse}). Given an oriented Riemannian $4$-manifold $(M,g)$ with Riemann curvature $4$-tensor
$$
\Rmr(v,w,x,y)=g(\cd{v}{\cd{w}{x}}-\cd{w}{\cd{v}{x}}-\cd{[v,w]}{x},y),
$$
recall that $g$'s \emph{curvature operator} $\co\colon\Lambda^2\lra\Lambda^2$ is the linear map defined by
\beqa
\label{eqn:co}
\ipr{\co(v\wedge w)}{x\wedge y}\defeq\Rmr(v,w,x,y)\hspace{.2in}\text{for all $v,w,x,y\in T_pM$},
\eeqa
where $\ipr{\,}{}$ is the $g$-induced inner product on the ${4\choose 2}=6$-dimensional second exterior product $\Lambda^2$:
\beqa
\label{eqn:g2}
\ipr{\ww{v}{w}}{\ww{x}{y}}\defeq \det\begin{bmatrix}g(v,x)&g(v,y)\\g(w,x)&g(w,y)\end{bmatrix}\cdot
\eeqa
Note that the pairwise symmetry $\Rmr(v,w,x,y)=\Rmr(x,y,v,w)$ ensures that $\co$ is $\ipr{\,}{}$-self-adjoint; note also that $\co$ satisfies the algebraic Bianchi identity with respect to $\ipr{\,}{}$. Had \eqref{eqn:co} been defined with a minus sign, then $\ipr{\co(v\wedge w)}{v\wedge w}=\Rmr(v,w,w,v)$ would be the sectional curvature of the $2$-plane $P\defeq\ww{v}{w}$ when $v,w$ are orthonormal (cf.~the minus signs in \eqref{eqn:Rgh-} and \eqref{eqn:Rg-} below). Thus the quadratic form of \eqref{eqn:co} plays an important role: if we denote by $\mathrm{Gr}_2(T_pM)$ the Grassmannian of $2$-planes in $T_pM$ (identified with the $\ipr{\,}{}$-unit decomposable bivectors in $\Lambda^2(T_pM)$), then the \emph{quadratic form of $\co$ at $p$} is the  function $\Rsec\colon\mathrm{Gr}_2(T_pM)\lra\RR$ defined by
\beqa
\label{eqn:quadR}
\Rsec(P)\defeq\ipr{\co P}{P}\comma \text{for all $P\in \mathrm{Gr}_2(T_pM)$}.
\eeqa
The second linear endomorphism on $\Lambda^2$ we shall need is the \emph{Hodge star operator} $\hsr\colon\Lambda^2\lra\Lambda^2$, defined via the equation
\beqa
\label{eqn:*def4}
\ww{\xi}{\hsr\eta}\defeq\ipr{\xi}{\eta}\,dV\comma \xi,\eta\in\Lambda^2,
\eeqa
where $dV\in\Lambda^4$ is the orientation form. By the classical work \cite{berger,thorpe2}, Einstein metrics are precisely those for which $\hsr$ commutes with $\co$:
\beqa
\label{eqn:var0}
\hsr\co=\co\!\hsr.
\eeqa
We now deepen the investigation begun in \cite{aazami}, namely, examining the geometry that arises when the Hodge star $\hsh$ arises from a different Riemannian metric $h$ on $M$. The key difference with \cite{aazami} is that here we work with a form of $\co$ that is adapted to $h$ (see \eqref{eqn:coh2}), or else with the product $\hsr\co$ itself (see Proposition \ref{prop:AJ} in Section \ref{sec:four}).
\vskip 6pt
Thus, given $(M,g)$, let $h$ be any other Riemannian metric on $M$ and denote by $\iph{\,}{}$ its corresponding inner product on $\Lambda^2$. If we use $\iph{\,}{}$ to express the symmetric bilinear form $\Rmr\colon\Lambda^2\times\Lambda^2\lra\RR$ as an endomorphism ``$\coh\colon\Lambda^2\lra\Lambda^2$,'' then its action is defined analogously to \eqref{eqn:co} as follows:
\beqa
\label{eqn:coh2}
\iph{\coh(v\wedge w)}{x\wedge y}\defeq\Rmr(v,w,x,y)\hspace{.2in}\text{for all $v,w,x,y\in T_pM$}.
\eeqa
($\Rmr$ is still the curvature $4$-tensor of $g$, not $h$.) Furthermore, just as with \eqref{eqn:quadR}, at each $p\in M$ we may associate to $\coh$ its $\iph{\,}{}$-quadratic form:
$$
\ghsec(P)\defeq\iph{\coh P}{P}\comma \text{for all $P\in \mathrm{Gr}_2(T_pM)$}.
$$
Observe that $\ghsec\neq\Rsec$, because the domains of the two quadratic forms are not equal: the set of $\iph{\,}{}$-unit decomposable bivectors is not the same as the set of the $\ipr{\,}{}$-unit ones. Indeed, if $P=\ww{v}{w}$ for an $h$-orthonormal pair $v,w$, then $\ww{v}{w}$ is in the domain of $\ghsec$, hence $\ghsec(P)=\Rmr(v,w,v,w)$; on the other hand, $\frac{\ww{v}{w}}{\sqrt{g(v,v)g(w,w)-g(v,w)^2}}$ is in the domain of $\Rsec$, so that
$$
\Rsec(P)=\frac{\Rmr(v,w,v,w)}{g(v,v)g(w,w)-g(v,w)^2}=\frac{\ghsec(P)}{g(v,v)g(w,w)-g(v,w)^2}\cdot
$$
Although $\ghsec(P)\neq\Rsec(P)$, the following is nevertheless true:

\begin{lemma}
\label{lemma:ghsec}
The quadratic form \emph{$\ghsec$} determines the curvature tensor of $g$.
\end{lemma}

\begin{proof}
The key is that, for any $g$ and $h$, $\coh$ is $\iph{\,}{}$-self-adjoint and satisfies the algebraic Bianchi identity with respect to it (by contrast, $\co$ does these with respect to $\ipr{\,}{}$). Now suppose that $\sigma_{\scalebox{0.6}{$\bar g$-$h$}}=\ghsec$, where $\bar g$ is another Riemannian metric on $M$. Then $\hat R_{\scalebox{0.6}{$\bar g$-$h$}}=\coh$ by the standard proof, since both endomorphisms are $\iph{\,}{}$-self-adjoint and satisfy the algebraic Bianchi identity with respect to it. Thus their symmetric bilinear forms $\text{Rm}_{\scalebox{0.6}{$\bar g$}},\Rmr\colon\Lambda^2\times\Lambda^2\lra\RR$ are equal.
\end{proof}

Next, define the mixed Ricci tensor and mixed scalar curvature of the pair $(g,h)$ by
$$
\Ricgh \defeq \mathrm{tr}_{\scalebox{0.6}{$h$}}\Rmr \qquad\text{and}\qquad \sgh\defeq \mathrm{tr}_{\scalebox{0.6}{$h$}}(\Ricgh ).
$$
In an oriented $h$-orthonormal frame $\{e_1,e_2,e_3,e_4\}$, and after setting \mbox{$\Rmr(e_i,e_j,e_k,e_l) \defeq R_{ijkl}$,}
$$
(\Ricgh )_{jl}=\sum_{i=1}^4R_{ijli}
\qquad\text{and}\qquad
\sgh=\sum_{j=1}^4(\Ricgh )_{jj}.
$$
These generalize the familiar $\Ric=\lambda g$ definition of the Einstein condition. We now present several equivalent formulations of $\Rmr(\Ls{h}{+},\Ls{h}{-})=0$.

\begin{prop}
\label{prop:traceequiv}
Let $(M,g)$ be an oriented Riemannian $4$-manifold and let $h$ be an auxiliary Riemannian metric. Then the following are equivalent:
\begin{enumerate}
\item[1.] \emph{$\Rmr(\Ls{h}{+},\Ls{h}{-})=0$.}
\item[2.] $\hsh\coh=\coh\hsh$.
\item[3.] \emph{$\ghsec(P)=\ghsec(\hsh P)$}.
\item[4.] \emph{$\Ricgh =\dfrac{\sgh}{4}h$}.
\end{enumerate}
\end{prop}

\begin{proof}
Fix $p\in M$. That 1.~is equivalent to 2.~follows from the fact that $\hsh\coh=\coh\hsh \iff \coh(\Ls{h}{\pm}) \subseteq \Ls{h}{\pm}$. Next, if $\hsh\coh=\coh\hsh$, then for any $2$-plane $P\in\mathrm{Gr}_2(T_pM)$,
$$
\underbrace{\iph{\coh(\hsh P)}{\hsh P}}_{\text{$\ghsec(\hsh P)$}}=\iph{\hsh(\coh P)}{\hsh P}=\underbrace{\iph{\coh P}{P}}_{\text{$\ghsec(P)$}},
$$
where we have used that $\hsh$ is $\iph{\,}{}$-self-adjoint and $\hsh^2=\mathrm{id}$. Conversely, if $\ghsec(P)=\ghsec(\hsh P)$ for all $P$, then
$$
\iph{(\hsh\coh\hsh)P}{P}=\underbrace{\iph{\coh(\hsh P)}{\hsh P}}_{\text{$\ghsec(\hsh P)$}}=\underbrace{\iph{\coh P}{P}}_{\text{$\ghsec(P)$}}.
$$
Thus $\hsh\coh\hsh$ and $\coh$ have equal $\iph{\,}{}$-quadratic forms.  Because they both satisfy the algebraic Bianchi identity with respect to $\iph{\,}{}$, it follows by the usual proof that they must be equal as endomorphisms $\Lambda^2\lra\Lambda^2$. This proves the equivalence of 2.~and 3. Next, choose an $h$-orthonormal basis $\{e_1,e_2,e_3,e_4\} \subseteq T_pM$. This determines the $\iph{\,}{}$-orthonormal basis
\beqa
\label{eqn:basis}
\{\ww{e_1}{e_2},\ww{e_1}{e_3},\ww{e_1}{e_4},\ww{e_3}{e_4},\ww{e_4}{e_2},\ww{e_2}{e_3}\}\subseteq\Lambda^2(T_pM)
\eeqa
with respect to which
$$
\coh=\begin{bmatrix}
R_{1212}&R_{1312}&R_{1412}&R_{3412}&R_{4212}&R_{2312}\\
R_{1213}&R_{1313}&R_{1413}&R_{3413}&R_{4213}&R_{2313}\\
R_{1214}&R_{1314}&R_{1414}&R_{3414}&R_{4214}&R_{2314}\\
R_{1234}&R_{1334}&R_{1434}&R_{3434}&R_{4234}&R_{2334}\\
R_{1242}&R_{1342}&R_{1442}&R_{3442}&R_{4242}&R_{2342}\\
R_{1223}&R_{1323}&R_{1423}&R_{3423}&R_{4223}&R_{2323}
\end{bmatrix}=\begin{bmatrix}A&B\\B^t&D\end{bmatrix},
$$
while
$$
\hsh=\begin{bmatrix}O&I\\I&O\end{bmatrix}\cdot
$$
Hence
$$
\hsh\coh=\coh\hsh \iff B^t=B\commass D=A.
$$
On the other hand, since $\Rmr$ is an algebraic curvature tensor, its $h$-Ricci decomposition has the usual $4$-dimensional form
\beqa
\label{eqn:Rgh-}
\coh=-\!
\begin{bmatrix}
W^+_{\scalebox{0.6}{$g$-$h$}}+\dfrac{\sgh}{12}I & K_{\scalebox{0.6}{$g$-$h$}}\\[4pt]
K_{\scalebox{0.6}{$g$-$h$}}^{t} & W^-_{\scalebox{0.6}{$g$-$h$}}+\dfrac{\sgh}{12}I
\end{bmatrix}
\eeqa
with respect to the splitting $\Lambda^2=\Ls{h}{+}\oplus\Ls{h}{-}$; here the block $K_{\scalebox{0.6}{$g$-$h$}}$, which satisfies $\begin{bmatrix}O&K_{\scalebox{0.6}{$g$-$h$}}\\K_{\scalebox{0.6}{$g$-$h$}}^t&O\end{bmatrix}=-\frac{1}{2}(\coh-\hsh \coh \hsh)$, is the image of the $h$-trace-free part
$$
\mathring{\text{Ric}}_{\scalebox{0.6}{$g$-$h$}}\defeq \Ricgh -\frac{\sgh}{4}h
$$
under the $O(T_pM,h)$-equivariant identification
$
S^2_0(T_pM,h)\cong \operatorname{Hom}(\Ls{h}{+},\Ls{h}{-}).
$
Therefore the commuting relation $\hsh\coh=\coh\hsh$ is equivalent to the vanishing of the off-diagonal block $K_{\scalebox{0.6}{$g$-$h$}}$, which in turn is equivalent to $\mathring{\text{Ric}}_{\scalebox{0.6}{$g$-$h$}}=0$. Thus
$$
\hsh\coh=\coh\hsh
\iff
\Ricgh =\frac{\sgh}{4}h.
$$
This proves the equivalence of $1.$ and $4.$, and completes the proof.
\end{proof}
Proposition \ref{prop:traceequiv} shows that the condition $\Rmr(\Ls{h}{+},\Ls{h}{-})=0$ is equivalent to the pure-trace equation $\Ricgh =\frac{\sgh}{4}h.$ We now close this section by showing that the latter equation opens the door to a variational characterization of $\Rmr(\Ls{h}{+},\Ls{h}{-})=0$; afterwards, we will comment on the difference between \eqref{eqn:Ag} below and the classical Einstein-Hilbert functional.

\begin{prop}
\label{thm:variational}
Fix a compact oriented Riemannian $4$-manifold $(M,g)$. For an auxiliary Riemannian metric $h$, define the functional
\emph{\beqa
\label{eqn:Ag}
\mathcal{A}_{\scalebox{0.6}{$g$}}(h)\defeq \int_M \sgh\,\dVh.
\eeqa}Then the Euler-Lagrange equation of $\mathcal{A}_{\scalebox{0.6}{$g$}}$ with respect to variations of $h$ \emph{(}with $g$ fixed\emph{)} is
\emph{\beqa
\label{eqn:EL}
\Ricgh =\frac{\sgh}{4}h.
\eeqa}Consequently, the critical points of $\mathcal{A}_{\scalebox{0.6}{$g$}}$ are exactly the auxiliary metrics $h$ for which \emph{$\Rmr(\Ls{h}{+},\Ls{h}{-})=0$.}
\end{prop}

\begin{proof}
Let $h_t\defeq h+tv$ be a smooth variation of $h$, where $v$ is a symmetric $(0,2)$-tensor, and keep $g$ fixed; note that $h_t$ remains Riemannian for small $|t|$. Note also that $\Rmr$ does not vary, only the $h$-contractions and the $h$-volume form do (in contrast to the Einstein-Hilbert action). Differentiating $h_t^{ik}(h_t)_{kj} = \delta_{ij}$, and using that $\frac{d}{dt}\big|_{t=0}\text{det}(h+tv) = \text{det}(h)\text{tr}(h^{-1}v)$, yields
$$
\left.\frac{d}{dt}\right|_{t=0}h_t^{ij}=-v^{ij}
\qquad\text{and}\qquad
\left.\frac{d}{dt}\right|_{t=0}dV_{\scalebox{0.6}{$h_t$}}=\frac12\,\mathrm{tr}_{\scalebox{0.6}{$h$}}(v)\,\dVh.
$$
Using the former, we now differentiate $\text{scal}_{\scalebox{0.6}{$g$-$h_t$}}= h_t^{jk}h_t^{il}R_{ijkl}$ at $t=0$,
\begin{align*}
\left.\frac{d}{dt}\right|_{t=0}\text{scal}_{\scalebox{0.6}{$g$-$h_t$}}
&=
-v^{jk}h^{il}R_{ijkl}-h^{jk}v^{il}R_{ijkl}\\
&=-2v^{jk}(\Ricgh )_{jk},
\end{align*}
where in the last step we used the symmetry $R_{ijkl}=R_{jilk}$. Hence
\beqa
\left.\frac{d}{dt}\right|_{t=0}\mathcal{A}_{\scalebox{0.6}{$g$}}(h_t) \!\!&=&\!\!
\int_M\left(\left.\frac{d}{dt}\right|_{t=0}\text{scal}_{\scalebox{0.6}{$g$-$h_t$}}\right)\,\dVh+
\int_M \sgh \left.\frac{d}{dt}\right|_{t=0}dV_{\scalebox{0.6}{$h_t$}}\nonumber\\
&=&\!\!
\int_M\Big(\!\!-\!2v^{jk}(\Ricgh )_{jk}+\frac12\sgh\,\mathrm{tr}_{\scalebox{0.6}{$h$}}(v)\Big)\,\dVh\nonumber\\
&=&\!\!
\int_M\Big\langle\!\!-\!2\Ricgh +\frac12\sgh h,\,v\Big\rangle_{\!\scalebox{0.6}{$h$}}\,\dVh.\nonumber
\eeqa
Since $v$ is arbitrary, the Euler-Lagrange equation is
$$
-2\,\Ricgh +\frac12\sgh h=0,
$$
which is exactly \eqref{eqn:EL}. The final claim now follows from Propositions \ref{prop:traceequiv}.
\end{proof}

Let us make two remarks regarding Proposition \ref{thm:variational}. First, it is enough to consider variations of the form $h_t=h+tv$ as we have done. Indeed, the space of smooth Riemannian metrics is an open subset of the vector space $\Gamma(S^2T^*M)$, so its tangent space at $h$ identifies with the full space of smooth symmetric $(0,2)$-tensors $v$. Hence every such $v$ determines a tangent direction at $h$, and, since positive-definiteness is an open condition, the linear path $h_t=h+tv$ remains Riemannian for all sufficiently small $|t|$. Moreover, the first variation of $\mathcal{A}_{\scalebox{0.6}{$g$}}$ at $h$ depends only on the initial velocity $\dot h_0=v$. Therefore vanishing of the first variation for all smooth variations through $h$ is equivalent to vanishing for all linear variations $h_t=h+tv$, and these suffice to obtain the Euler-Lagrange equation. Second, although
$$
\mathcal{A}_{\scalebox{0.6}{$g$}}(g)=\int_M \text{scal}_{\scalebox{0.6}{$g$}}\,\dVg
$$
recovers the Einstein-Hilbert functional, the two variational problems are quite different. In the latter one varies a metric that determines both the curvature tensor and the volume form, so the first variation requires differentiating the Levi-Civita connection and curvature; as is well known, this yields a second-order elliptic PDE. Here, by contrast, $g$ is fixed throughout, and only the auxiliary metric $h$ is varied; as mentioned above, $\Rmr$ does not change. Moreover, unlike the ordinary Einstein-Hilbert functional, no separate volume normalization is needed here: in dimension four, $\mathcal{A}_{\scalebox{0.6}{$g$}}$ is already scale-invariant in $h$, and thus depends only on the oriented conformal class $[h]$. Indeed, if $\tilde h=e^{2u}h$, then $\tilde h^{-1}=e^{-2u}h^{-1}$ and $dV_{\scalebox{0.6}{$\tilde{h}$}}=e^{4u}\dVh$, so that
$
\text{scal}_{\scalebox{0.6}{$g$-$\tilde{h}$}}=e^{-4u}\sgh
$
and hence
$$
\text{scal}_{\scalebox{0.6}{$g$-$\tilde{h}$}}\,dV_{\scalebox{0.6}{$\tilde{h}$}}=\sgh\,\dVh.
$$
This is compatible with the fact that $\hsh$ on $2$-forms is conformally invariant in dimension four.

\section{Topological obstructions to $\Rmr(\Ls{h}{+},\Ls{h}{-})=0$}
\label{sec:four}
We now return to the endomorphism $\cow$ mentioned in the Introduction. Recall that in dimension four the Hodge star $\hsh$ depends only on the oriented conformal class of $h$, so that choosing $h$ is equivalent to choosing an oriented conformal class $[h]$. This motivates the following reformulation of $\Rmr(\Ls{h}{+},\Ls{h}{-})=0$; as we shall see, it will yield a precise algebraic obstruction to $\Rmr(\Ls{h}{+},\Ls{h}{-})=0$, which in turn will yield a topological obstruction that will generalize the Hitchin--Thorpe inequality.

\begin{prop}
\label{prop:AJ}
Fix $\Lambda^2\defeq\Lambda^2(T_pM)$. Let $Q$ be the wedge pairing on $\Lambda^2$,
\beqa
\label{eqn:Qform}
Q(\xi,\eta)\dVg \defeq \ww{\xi}{\eta} \comma \xi,\eta \in \Lambda^2,
\eeqa
and let \emph{$\Rmr(\xi,\eta)$} denote the curvature bilinear form of $g$ on $\Lambda^2$. Then the $Q$-self-adjoint endomorphism $\cow\colon\Lambda^2\lra\Lambda^2$ defined by the equation \emph{$
Q(\cow(\xi),\eta) \defeq \Rmr(\xi,\eta)
$} satisfies $$\cow = \hsr\co.$$ Furthermore, for an auxiliary Riemannian metric $h$, the following are equivalent at $p$:
\begin{enumerate}
\item[1.] \emph{$\Rmr(\Ls{h}{+},\Ls{h}{-})=0$}.
\item[2.] $\cow(\Ls{h}{\pm})\subseteq\Ls{h}{\pm}$.
\item[3.] $\cow \hsh=\hsh\cow$.
\end{enumerate}
\end{prop}

\begin{proof}
Since $\hsh$ has eigenspace decomposition $\Lambda^2=\Ls{h}{+}\oplus\Ls{h}{-}$, the equivalence of $2.$ and $3.$ is immediate. Next, because
\beqa
\label{eqn:Qhsh}
Q(\xi,\hsh\eta)\dVh\overset{\eqref{eqn:Qform}}{=}\ww{\xi}{\hsh\eta}\overset{\eqref{eqn:*def4}}{=}\iph{\xi}{\eta}\dVh,
\eeqa
it follows that, for $\xi_{\pm}\in\Ls{h}{\pm}$,
$$
Q(\xi_\pm,\xi_\pm)\dVh = \pm Q(\xi_\pm,\hsh\xi_\pm)\dVh = \pm\iph{\xi_\pm}{\xi_\pm}\dVh.
$$
Thus $\Ls{h}{+}$ is $Q$-positive, $\Ls{h}{-}$ is $Q$-negative, and the splitting is $Q$-orthogonal (note that changing the volume form from $\dVg$ to $\dVh$ rescales $Q$ by a positive factor, preserving positive and negative subspaces). Furthermore, by construction we have
$
Q(\cow (\xi_+),\xi_-) = \Rmr(\xi_+,\xi_-).
$
Because $\Ls{h}{-}=(\Ls{h}{+})^{\perp_{\scalebox{0.5}{$Q$}}}$, the vanishing of all such mixed terms is equivalent to $\cow(\xi_+)$ being $Q$-orthogonal to $\Ls{h}{-}$, hence to $\cow(\xi_+)\in\Ls{h}{+}$. Likewise on $\Ls{h}{-}$. This proves the equivalence of $1.$ and $2$. Finally, $\cow = \hsr\co$ follows from
$$
\ipr{\co(\xi)}{\eta} = \Rmr(\xi,\eta) = Q(\cow(\xi),\eta) = \ipr{\cow(\xi)}{\hsr\eta} = \ipr{\hsr\cow(\xi)}{\eta},
$$
with the second-to-last equality by an argument as in \eqref{eqn:Qhsh}.
\end{proof}

In fact the question of whether $\Rmr(\Ls{h}{+},\Ls{h}{-})=0$ is possible can be recast in terms of the action of $\mathrm{SL}(4,\RR)$ on $\Lambda^2$.  Indeed, recall that the latter acts on $\Lambda^2$ via $A(\ww{e_i}{e_j}) = \ww{(Ae_i)}{(Ae_j)}$, and satisfies
$$
Q(A\xi,A\eta)dV = \ww{(A\xi)}{(A\eta)} = A(\ww{\xi}{\eta})=(\text{det}A)\,\ww{\xi}{\eta}=Q(\xi,\eta)dV.
$$
Thus $\text{SL}(4,\RR)$ preserves $Q$; in fact it acts transitively on $Q$-positive $3$-planes. Indeed, $\text{GL}^+(4,\RR)$ acts transitively on such planes, since every $Q$-positive $3$-plane equals $\Ls{h}{+}$ for some unique oriented conformal class $[h]$ (see \cite[p.~8]{Donaldson}); after multiplying any such $A\in\text{GL}^+(4,\RR)$ by a positive scalar, we may arrange that $\det A=1$, and this scalar multiplication does not change the image of a $3$-plane in $\Lambda^2$. Thus, if we fix a Riemannian metric $h_0$ with Hodge splitting $\Ls{0}{+}\oplus\Ls{0}{-}$, then the question of whether $\Rmr(\Ls{h}{+},\Ls{h}{-})=0$ holds for some $[h]$ is equivalent to whether there exists $A \in \mathrm{SL}(4,\RR)$ such that $\Rmr(A\Ls{0}{+},A\Ls{0}{-})=0$. As we now show, this pointwise algebraic question can be answered completely, by giving the following exact existence criterion for $\Rmr(\Ls{h}{+},\Ls{h}{-})=0$. Indeed, because $Q$ has signature $(3,3)$, note that $\cow$ need not be real-diagonalizable (i.e., diagonalizable as an endomorphism of the real vector space $\Lambda^2(T_pM)$). It turns out that this is precisely the algebraic obstruction to $\Rmr(\Ls{h}{+},\Ls{h}{-})=0$.

\begin{prop}
\label{thm:diagonalizable}
The $Q$-self-adjoint map $\cow$ commutes with the Hodge star $\hsh$ of some Riemannian metric $h$ on $T_pM$ if and only if $\cow$ is real-diagonalizable.
\end{prop}

\begin{proof}
Suppose first that $\cow \hsh=\hsh\cow$ for the Hodge star $\hsh$ of a Riemannian metric $h$ on $T_pM$.  Then $q_{\scalebox{0.6}{$h$}}(\xi,\eta)\defeq Q(\xi,\hsh\eta)$
is positive-definite, and
$$
q_{\scalebox{0.6}{$h$}}(\cow\xi,\eta)=Q(\cow\xi,\hsh\eta)=Q(\xi,\cow\!\hsh\eta)=Q(\xi,\hsh\cow\eta)=q_{\scalebox{0.6}{$h$}}(\xi,\cow\eta).
$$
Thus $\cow$ is self-adjoint with respect to the positive-definite inner \mbox{product} $q_{\scalebox{0.6}{$h$}}$, hence is real-diagonalizable. Conversely, assume that $\cow$ is real-diagonalizable, and write
$
\Lambda^2(T_pM)=\oplus_{\lambda}E_{\lambda}
$
for the direct sum of its real eigenspaces.  If $\lambda\neq\mu$, then $E_\lambda$ and $E_\mu$ are $Q$-orthogonal, by the standard argument; it follows that each $E_\lambda$ is $Q$-nondegenerate.  Now, for each $\lambda$, let $E_\lambda^{\pm}$ denote a choice of maximal $Q$-positive and -negative subspaces of $E_\lambda$ satisfying $E_\lambda = E_\lambda^+\oplus E_\lambda^-$; in particular, $E_\lambda^-= (E_\lambda^+)^{\perp_{Q|_{\scalebox{0.45}{$E_\lambda$}}}}$. Define an involution $J_\lambda$ on each $(E_\lambda,Q|_{E_\lambda})$ by
$$
J_\lambda\big|_{E_\lambda^+}\defeq+1\comma J_\lambda\big|_{E_\lambda^-}\defeq-1,
$$
and set \mbox{$J\defeq\oplus_\lambda J_\lambda$.}  Then $J$ is an involution on $\Lambda^2(T_pM)$. Furthermore, $J$ is $Q$-self-adjoint because the eigenspace splitting $\Lambda^2=J^+\oplus J^-$ is $Q$-orthogonal by construction, with each $J^{\pm}$ $3$-dimensional because $Q$ has signature $(3,3)$. Moreover, $J$ preserves each eigenspace of $\cow$ by construction, so \mbox{$\cow J=J\cow$.} We now recall that every $Q$-positive 3-plane equals $\Ls{h}{+}$ for some unique conformal class $[h]$. Applying this to $J^+$ gives an oriented conformal class $[h]$ with $\Ls{h}{+}=J^+$. Since both $J^-$ and $\Ls{h}{-}$ are the $Q$-orthogonal complement of $J^+$, they agree as well. Thus, for any representative of $[h]$, its Hodge involution satisfies $\hsh=J$.
\end{proof}

Thus, at a fixed point, if any $E_\lambda$ is $Q$-indefinite, then there are infinitely many compatible oriented conformal classes $[h]$ satisfying \mbox{$\Rmr(\Ls{h}{+},\Ls{h}{-})=0$;} otherwise the compatible conformal class is unique. The more important question is what topological obstructions arise from this.

\begin{cor}[Euler obstruction]
\label{cor:generalchi}
Let $(M,g)$ be a compact oriented Riemannian $4$-manifold.  If there exists a Riemannian metric $h$ such that
\emph{$\Rmr(\Ls{h}{+},\Ls{h}{-})=0$},
then $\chi(M)\ge0$.  Equality holds if and only if $g$ is flat.
\end{cor}

\begin{proof}
By Proposition \ref{prop:AJ}, $\hsh$ commutes with $\cow$. Because $\cow = \hsr\co$, and given the usual decompositions
\beqa
\label{eqn:Rg-}
\co=-\!
\begin{bmatrix}
W^+_{\scalebox{0.6}{$g$}}+\dfrac{\text{scal}_{\scalebox{0.6}{$g$}}}{12}I & \text{Ric}_{\scalebox{0.6}{$0$}}\\[4pt]
\text{Ric}_{\scalebox{0.6}{$0$}}^{t} & W^-_{\scalebox{0.6}{$g$}}+\dfrac{\text{scal}_{\scalebox{0.6}{$g$}}}{12}I
\end{bmatrix} \comma \hsr = \begin{bmatrix}I&O\\O&-I\end{bmatrix}
\eeqa
with respect to the $\hsr$-splitting $\Lambda^2=\Ls{g}{+}\oplus\Ls{g}{-}$, we have
$$
\mathrm{tr}(\cow^2) = |W^+_{\scalebox{0.6}{$g$}}|^2 + |W^-_{\scalebox{0.6}{$g$}}|^2 - \frac{1}{2}|\text{Ric}_{\scalebox{0.6}{$0$}}|^2 + \frac{\text{scal}^2_{\scalebox{0.6}{$g$}}}{24}\cdot
$$
This is precisely the Chern--Gauss--Bonnet integrand, so that
$$
\chi(M)=\frac{1}{8\pi^2}\int_M \mathrm{tr}(\cow^2)\,\dVg.
$$
Since $\cow$ is real-diagonalizable by Proposition \ref{thm:diagonalizable},
$
\mathrm{tr}(\cow^2)=\sum_{i=1}^6\lambda_i^2\geq0,
$
so that $\chi(M)\ge0$.  If equality holds, then each $\lambda_i=0$, hence $\cow=0$, hence $\Rmr=0$ and $g$ is flat.  The converse is immediate.  
\end{proof}

\begin{thm}[Relative Hitchin--Thorpe inequality]
\label{thm:relativeHT}
Let $(M,g)$ be a compact oriented Riemannian $4$-manifold, and let $h$ be a Riemannian metric satisfying \emph{$\Rmr(\Ls{h}{+},\Ls{h}{-})=0$}. Set $q_{\scalebox{0.6}{$h$}}(\xi,\eta)\defeq Q(\xi,\hsh\eta)$ and 
$$
\rhoz(g,h)
\defeq
\sup_M\left\|\frac12\left(\hsr+\hsh \hsr\hsh\right)\right\|_{\text{$\qh$}}\cdot
$$
Then $\rhoz(g,h) \geq1$ and
\beqa
\label{eqn:relativeHT}
2\chi(M)\ge \frac{3}{\rhoz(g,h)}|\tau(M)|.
\eeqa
Furthermore, $\rhoz(g,h)=1$ if and only if $h \in [g]$, in which case the inequality reduces to the classical Hitchin--Thorpe inequality.
\end{thm}

\begin{proof}
By Proposition \ref{prop:AJ}, $\cow\hsh=\hsh\cow$.  By Proposition \ref{thm:diagonalizable}, $\cow$ is \mbox{$q_{\scalebox{0.6}{$h$}}$-self-adjoint,} hence real-diagonalizable; thus $\cow^2$ is $q_{\scalebox{0.6}{$h$}}$-self-adjoint and positive-semidefinite. However, $g$'s Hodge star $\hsr$ need not be $q_{\scalebox{0.6}{$h$}}$-self-adjoint; indeed, using that $\hsr$ is $Q$-self-adjoint, we have
$$
q_{\scalebox{0.6}{$h$}}(\hsr\xi,\eta)
=Q(\hsr\xi,\hsh\eta)
=Q(\xi,\hsr\!\hsh\!\eta)
=q_{\scalebox{0.6}{$h$}}(\xi,\hsh\!\hsr\!\hsh\eta),
$$
from which it follows that the $q_{\scalebox{0.6}{$h$}}$-self-adjoint part of $\hsr$ is
$$
H_{\scalebox{0.6}{$g,h$}}
\defeq
\frac{1}{2}\left(\hsr+\hsh\!\hsr\!\hsh\right).
$$
Because $\cow\hsh=\hsh\cow$,
$$
\mathrm{tr}(\hsh\!\hsr\!\hsh\cow^2)=\mathrm{tr}(\hsr\!\hsh\!\cow^2\hsh)=\mathrm{tr}(\hsr\!\hsh\!\hsh\cow^2)=\mathrm{tr}(\hsr\cow^2),
$$
hence
$
\mathrm{tr}(H_{\scalebox{0.6}{$g,h$}} \cow^2)=\mathrm{tr}(\hsr \cow^2).
$
Choose a $q_{\scalebox{0.6}{$h$}}$-orthonormal eigenbasis $\{\xi_i\}_{i=1}^6$ for $\cow^2$, say $\cow^2 \xi_i=\mu_i \xi_i$ with each $\mu_i\ge0$. With respect to this basis,
$$
\cow^2 = \mathrm{diag}(\mu_1,\dots,\mu_6) \comma H_{\scalebox{0.6}{$g,h$}} =H_{\scalebox{0.6}{$g,h$}}^t= (q_{\scalebox{0.6}{$h$}}(H_{\scalebox{0.6}{$g,h$}}\xi_i,\xi_j)).
$$
In particular, each $|q_{\scalebox{0.6}{$h$}}(H_{\scalebox{0.6}{$g,h$}}\xi_i,\xi_i)|\leq \|H_{\scalebox{0.6}{$g,h$}}\|_{\text{$\qh$}}\leq \rhoz(g,h)$, so that
$$
\big|\mathrm{tr}(H_{\scalebox{0.6}{$g,h$}} \cow^2)\big|=\left|\sum_{i=1}^6 \mu_i\,q_{\scalebox{0.6}{$h$}}(H_{\scalebox{0.6}{$g,h$}}\xi_i,\xi_i)\right|\leq \rhoz(g,h)\sum_{i=1}^6\mu_i=\rhoz(g,h)\,\mathrm{tr}(\cow^2).
$$
Thus
$
\big|\mathrm{tr}(\hsr\cow^2)\big|
\leq\rhoz(g,h)\,\mathrm{tr}(\cow^2).
$
Finally, using the Chern--Gauss--Bonnet and signature formulas in the form
$$
\chi(M)=\frac{1}{8\pi^2}\int_M\mathrm{tr}(\cow^2)\,\dVg \comma 
\tau(M)=\frac{1}{12\pi^2}\int_M\mathrm{tr}(\hsr\cow^2)\,\dVg
$$
(the latter because $\mathrm{tr}(\hsr\cow^2) = \mathrm{tr}(\hsr\co^2)=|W^+_{\scalebox{0.6}{$g$}}|^2 - |W^-_{\scalebox{0.6}{$g$}}|^2$), we obtain
$$
|\tau(M)|
\le
\frac{\rhoz(g,h)}{12\pi^2}\int_M\mathrm{tr}(\cow^2)\,\dVg
=
\frac{2\rhoz(g,h)}{3}\chi(M).
$$
If $h\in[g]$, then $\hsh=\hsr$ and hence
$
\frac{1}{2}(\hsr+\hsh\hsr\hsh)=\hsr.
$
Since $\hsr$ is an isometry of $q_{\scalebox{0.6}{$g$}}$, one has $\rhoz=1$.  This gives Hitchin--Thorpe. To complete the proof, we show that $\rhoz \geq 1$ and equal to $1$ if and only if $h \in [g]$. First, any $\hsh$-eigenbasis $\{\xi_i^+,\xi_i^-\}_{i=1}^3$ remains orthonormal with respect to $q_{\scalebox{0.6}{$h$}}$, and furthermore the latter preserves the $\hsh$-splitting $\Lambda^2 = \Ls{h}{+}\oplus \Ls{h}{-}$: 
$$
q_{\scalebox{0.6}{$h$}}(\xi^+,\xi^-) = -Q(\xi^+,\xi^-) = -Q(\xi^-,\xi^+) = -Q(\xi^-,\hsh \xi^+) = -q_{\scalebox{0.6}{$h$}}(\xi^-,\xi^+).
$$
Express $\hsh, \hsr$, and then $H_{\scalebox{0.6}{$g,h$}}$ with respect to $\{\xi_i^+,\xi_i^-\}_{i=1}^3$:
$$
\hsh = \begin{bmatrix}I&O\\O&-I\end{bmatrix} \commas \hsr = \begin{bmatrix}A&B\\C&D\end{bmatrix} \imp H_{\scalebox{0.6}{$g,h$}} = \begin{bmatrix}A&O\\O&D\end{bmatrix}\cdot
$$
Next, observe that the projection of $\Ls{g}{+} \subseteq \Lambda^2 = \Ls{h}{+}\oplus \Ls{h}{-}$ onto $\Ls{h}{+}$ is injective and hence an isomorphism of 3-planes; this allows us to ``go from $\Ls{h}{+}$ to $\Ls{h}{-}$ through $\Ls{g}{+}$"; concretely, to define a linear map $T\colon \Ls{h}{+} \lra \Ls{h}{-}$ by
$$
\xi^+ \in \Ls{h}{+} \mapsto\,\text{unique}~\xi^+ + \xi^- \in \Ls{g}{+} \mapsto \xi^- \in \Ls{h}{-}.
$$
In terms of this map, $\Ls{g}{+} = \{\xi^+ + T(\xi^+) : \xi^+ \in \Ls{h}{+}\}$. Furthermore, because $\Ls{g}{+}$ is $Q$-positive, and because $Q$ also preserves the $\hsh$-splitting $\Lambda^2=\Ls{h}{+}\oplus \Ls{h}{-}$, we have that
$$
Q(\xi^+ + T(\xi^+),\xi^+ + T(\xi^+)) = q_{\scalebox{0.6}{$h$}}(\xi^+,\xi^+) - q_{\scalebox{0.6}{$h$}}(T(\xi^+),T(\xi^+)) > 0.
$$
For any $q_{\scalebox{0.6}{$h$}}$-unit $\xi^+ \in \Ls{h}{+}$, we thus have that
$$
q_{\scalebox{0.6}{$h$}}(T(\xi^+),T(\xi^+)) < q_{\scalebox{0.6}{$h$}}(\xi^+,\xi^+) = 1 \imp \|T\|_{\scalebox{0.6}{$q_{\scalebox{0.6}{$h$}}$}} < 1,
$$
by compactness of the $q_{\scalebox{0.6}{$h$}}$-unit sphere in $\Ls{h}{+}$ and continuity. Now consider the analogous construction on $\Ls{g}{-}$; this yields a linear map $T^*\colon\Ls{h}{-} \lra \Ls{h}{+}$, in terms of which $\Ls{g}{-} = \{\xi^- + T^*(\xi^-) : \xi^-\in \Ls{h}{-}\}$. In fact $T^*$ is the $q_{\scalebox{0.6}{$h$}}$-adjoint of $T$: using that $Q$ preserves both of the $\hsr$- and $\hsh$-splittings, we have
$$
Q(\xi^++T(\xi^+),\xi^-+T^*(\xi^-))=0 \imp q_{\scalebox{0.6}{$h$}}(T(\xi^+),\xi^-) = q_{\scalebox{0.6}{$h$}}(\xi^+,T^*(\xi^-)).
$$
Furthermore, $\|T^*\|_{\scalebox{0.6}{$q_{\scalebox{0.6}{$h$}}$}} = \|T\|_{\scalebox{0.6}{$q_{\scalebox{0.6}{$h$}}$}} < 1$, and for any $q_{\scalebox{0.6}{$h$}}$-unit eigenvector $\eta \in \Ls{h}{+}$ of the $q_{\scalebox{0.6}{$h$}}$-self-adjoint composition $T^*T$,
$$
T^*T\eta = \lambda \eta \imp \lambda = q_{\scalebox{0.6}{$h$}}(T^*T\eta,\eta) = q_{\scalebox{0.6}{$h$}}(T\eta,T\eta) \leq \|T\|_{\scalebox{0.6}{$q_{\scalebox{0.6}{$h$}}$}} < 1.
$$ 
Similarly for $TT^*$. Thus the eigenvalues of both $T^*T$ and $TT^*$ lie within $[0,1)$; in particular, $I-T^*T\colon \Ls{h}{+}\lra \Ls{h}{+}$ is invertible. We now express $A$ in terms of $I\pm T^*T$. Indeed, for $\xi^+ \in \Ls{h}{+}$ and $T(\xi^+) \in \Ls{h}{-}$, consider $(\xi^{+}+T(\xi^{+})) \in \Ls{g}{+}$ and $(T(\xi^+)+T^*(T(\xi^+))) \in \Ls{g}{-}$; then
\beqa
\hsr\big((\xi^{+}+T(\xi^{+}))-(T(\xi^+)+T^*(T(\xi^+)))\big) \!\!&=&\!\! \hsr(I-T^*T)(\xi^+)\nonumber\\
&=&\!\! (\xi^{+}+T(\xi^{+}))\nonumber\\
&&\hspace{.2in}+\,\big(T(\xi^+)+T^*(T(\xi^+))\big)\nonumber\\
&=&\!\! \underbrace{(I+T^*T)(\xi^+)}_{\text{in $\Ls{h}{+}$}}+\underbrace{2T(\xi^+)}_{\text{in $\Ls{h}{-}$}}.\nonumber
\eeqa
As $(I- T^*T)(\xi^+) \in \Ls{h}{+}$, it follows that $A = (I+T^*T)(I-T^*T)^{-1}$. Similarly, $D = -(I+TT^*)(I-TT^*)^{-1}$. Thus if $\lambda$ and $\mu$ are eigenvalues of $T^*T$ and $TT^*$, then
$\frac{1+\lambda}{1-\lambda} \geq 1$ and $-\frac{1+\mu}{1-\mu} \leq -1$
are eigenvalues of $A$ and $D$, respectively. Thus $\|H_{\scalebox{0.6}{$g,h$}}\|_{\scalebox{0.6}{$q_{\scalebox{0.6}{$h$}}$}} \geq 1$, hence the supremum $\rhoz(g,h) \geq 1$, with equality if and only if each $\lambda=\mu=0$; i.e., if and only if $T=T^*=0$, in which case $\Ls{g}{+} = \Ls{h}{+}$, hence $h \in [g]$.
\end{proof}

When $\hsh=\hsr$, the classical Hitchin--Thorpe inequality is recovered.  But as the $g$- and $h$-Hodge splittings separate, $\rhoz(g,h)$ may increase, weakening the estimate. We will show in Theorem \ref{thm:CP2example} that this may indeed happen.

\section{Examples}
\label{sec:examples}
We now record two families of pairs $(g,h)$ satisfying $\Rmr(\Ls{h}{+},\Ls{h}{-})=0$; both families have the property that $g$ and $h$ share orthogonal frames. A third example, also within our shared-orthogonal-frame ansatz, will follow afterwards in Section \ref{sec:cp2}.

\begin{prop}
\label{prop:product}
Let $(\Sigma_1,g_1)$ and $(\Sigma_2,g_2)$ be oriented closed Riemannian surfaces with Gauss curvatures $K_1$ and $K_2$, respectively. Assume that $K_1$ and $K_2$ are nowhere zero and have the same sign. Set
$$
(M,g)=(\Sigma_1\times\Sigma_2,g_1\oplus g_2).
$$
For positive smooth functions $\alpha,\beta$ on $M$, define the Riemannian metric
$$
h\defeq \alpha^2g_1\oplus\beta^2g_2.
$$
Then \emph{$\Rmr(\Ls{h}{+},\Ls{h}{-})=0$} if and only if
$$
\frac{K_1}{\alpha^4}=\frac{K_2}{\beta^4}\cdot
$$
Moreover, if $K_1\neq K_2$, then $g$ is not Einstein.
\end{prop}

\begin{proof}
Let $\{e_1,e_2\}$ be an oriented local $g_1$-orthonormal frame on $\Sigma_1$ and $\{e_3,e_4\}$ an oriented local $g_2$-orthonormal frame on $\Sigma_2$. Then $\{e_1,e_2,e_3,e_4\}$ is a local $g$-orthonormal frame on $M$. As the Levi-Civita connection of a Riemannian product splits, the only nonzero sectional curvatures of $g$ are
$$
\gsec(e_1\wedge e_2)=K_1\comma \gsec(e_3\wedge e_4)=K_2,
$$
and every mixed $2$-plane has sectional curvature $0$. Now define the $h$-orthonormal frame $\{\tilde e_1,\tilde e_2,\tilde e_3,\tilde e_4\}\defeq\{\alpha^{-1}e_1,\alpha^{-1}e_2,\beta^{-1}e_3,\beta^{-1}e_4\}$. With respect to the $h$-orthonormal basis $\{\ww{\tilde e_i}{\tilde e_j}\}$, $\coh$ takes the form
$$
\coh=-\text{diag}\!\left(\frac{K_1}{\alpha^4},0,0,\frac{K_2}{\beta^4},0,0\right),
$$
from which $\hsh\coh=\coh\hsh$ follows if and only if $\frac{K_1}{\alpha^4}=\frac{K_2}{\beta^4}$ (regarding the minus sign in $\coh$, recall our sign conventions in \eqref{eqn:co} and \eqref{eqn:coh2}). Finally, a product metric $g_1\oplus g_2$ is Einstein if and only if the two Ricci curvatures agree, which for surfaces means exactly $K_1=K_2$.
\end{proof}

For instance, one may take $\alpha=|K_1|^{1/4}/c$ and $\beta=|K_2|^{1/4}/c$ for any constant $c > 0$. Thus Proposition \ref{prop:product} immediately yields pairs $(g,h)$ on $\Sph^2\times\Sph^2$ with differently scaled round factors, as well as on $\Sigma_{\scalebox{0.6}{$g$}}\times\Sigma_{\scalebox{0.6}{$h$}}$ for closed hyperbolic surfaces of genera $g,h\ge2$ equipped with differently scaled hyperbolic metrics. Our next example is from among the warped product family.

\begin{prop}
\label{prop:warped}
Let $f\colon[0,\pi]\lra\RR$ be positive on $(0,\pi)$ and suppose that $f$ is a smooth warping function at the two poles; equivalently, near $0$ one has
$
f(t)=t+O(t^3)
$
with only odd powers in its Taylor expansion, and near $\pi$ one has the analogous expansion in the variable $\pi-t$, with leading term $\pi-t$.  On $(0,\pi)\times\Sph^3$, set
$$
g\defeq dt^2+f(t)^2g_{\scalebox{0.5}{$\Sph^3$}},
$$
and define
$$
A(t)\defeq-\frac{f''(t)}{f(t)}\comma B(t)\defeq \frac{1-(f'(t))^2}{f(t)^2}\cdot
$$
Assume $A(t)>0$ and $B(t)>0$ on $(0,\pi)$, and assume in addition that, after the arclength reparametrization
$
ds=\sqrt{\frac{A}{B}}\,dt,
$
the warping function $f(t(s))$ satisfies the same smooth oddness conditions at the two endpoints.  Then
$$
h\defeq\frac{A}{B}\,dt^2+f(t)^2g_{\scalebox{0.5}{$\Sph^3$}}
$$
extends smoothly to a Riemannian metric on $\Sph^4$, and \emph{$\Rmr(\Ls{h}{+},\Ls{h}{-})=0$}.  Moreover, $g$ is Einstein if and only if $A=B$, in which case $g$ is the round metric.
\end{prop}

\begin{proof}
The smoothness hypothesis on $f$ is the standard one for the warped product $dt^2+f^2g_{\scalebox{0.5}{$\Sph^3$}}$ (see, e.g., \cite[Sec.~1.4.4]{PP}), so $g$ extends smoothly over the two collapsed $\Sph^3$-orbits.  The additional arclength condition is precisely the same criterion applied to the auxiliary metric $h$: in the $s$-coordinate, $h=ds^2+f(t(s))^2g_{\scalebox{0.5}{$\Sph^3$}}$. Let $\{\bar e_2,\bar e_3,\bar e_4\}$ be a local orthonormal frame for $g_{\scalebox{0.5}{$\Sph^3$}}$ and set
$$
e_1\defeq\partial_t\comma  e_i\defeq f^{-1}\bar e_i\comma i=2,3,4.
$$
With respect to this $g$-orthonormal frame, the only sectional-curvature types are
$$
\gsec(e_1\wedge e_i)=A\comma
\gsec(e_i\wedge e_j)=B\comma i,j=2,3,4.
$$
With respect to $\{e_i\wedge e_j\}$, we thus have
$$
\co=-\text{diag}(A,A,A,B,B,B).
$$
In the corresponding $h$-orthonormal frame
$$
\tilde e_1\defeq\sqrt{\frac{B}{A}}e_1\comma
\tilde e_i\defeq e_i\comma i=2,3,4,
$$
one has
$
\Rmr(\tilde e_1,e_i,e_i,\tilde e_1)=\Rmr(e_i,e_j,e_j,e_i)=B,
$
and hence
$$
\coh=-\text{diag}(B,B,B,B,B,B).
$$
Therefore $\hsh\coh=\coh\hsh$.  Finally, $g$ is Einstein exactly when $A=B$; in the smooth rotationally symmetric $\Sph^4$ ansatz this is the round case.  
\end{proof}

For example, for the choice $f(t)\defeq\sin t\,(1+\vep\sin^2 t)$ with $0<\vep<\frac16$, all the conditions in Proposition \ref{prop:warped} are satisfied, with
$$
A(t)=\frac{1-6\vep+9\vep\sin^2 t}{1+\vep\sin^2 t}\commass B(t)=\frac{1-6\vep+(6\vep-9\vep^2)\sin^2 t+9\vep^2\sin^4 t}{(1+\vep\sin^2 t)^2}\cdot
$$
For this explicit family, $A/B$ is a smooth even function of the polar variable with value $1$ at the pole; hence $s(t)$ is an odd smooth function with $s'(0)=1$, and the inverse $t(s)$ is odd. Therefore $f(t(s))$ has the required odd expansion, so the smoothness criterion for $h$ is satisfied. Because $A\neq B$, the metric $g$ is not Einstein. Note that Propositions \ref{prop:product} and \ref{prop:warped} are both of shared-orthogonal-frame type. Such pairs $(g,h)$ are restricted by the following fact:

\begin{lemma}
\label{rem:sharedframe}
Let $\{e_1,e_2,e_3,e_4\} \subseteq T_pM$ be a $g$-orthonormal basis such that, with respect to the basis $\{\ww{e_i}{e_j}\} \subseteq \Lambda^2(T_pM)$, $\co$ has the matrix
$$
\co=\begin{bmatrix}A&B\\B&C\end{bmatrix}\commas \left\{\begin{array}{lcl}
A&=&\operatorname{diag}(a_1,a_2,a_3),\\
B&=&\operatorname{diag}(b_1,b_2,b_3),\\
C&=&\operatorname{diag}(c_1,c_2,c_3).\end{array}\right.
$$
If there is a second Riemannian metric $h$ for which $\{e_1,e_2,e_3,e_4\}$ is $h$-orthogonal, say $h(e_i,e_j)=x_i^{-1}\delta_{ij}$ with $x_i>0$, then \mbox{$\hsh\coh=\coh\hsh$} reduces to
$$
x_1x_2a_1=x_3x_4c_1\commas x_1x_3a_2=x_2x_4c_2\commas x_1x_4a_3=x_2x_3c_3.
$$
Hence, such an $h$ exists if and only if each complementary pair has the same sign, or both vanish.
\end{lemma}

\begin{proof}
This is a straightforward computation.
\end{proof}
In other words, within the shared-orthogonal-frame ansatz, a Riemannian auxiliary metric $h$ may rebalance the magnitudes of $\hsr$-opposite sectional curvatures, but not flip their signs.  This restriction is important; it implies $\chi\ge0$ but as we shall see, it does not imply Hitchin--Thorpe.

\section{A Hitchin--Thorpe-violating example}
\label{sec:cp2}
Although Corollary \ref{cor:generalchi} already gives $\chi(M)\ge0$ without any shared-frame hypothesis, the shared-frame formulas are useful because they show explicitly why Hitchin--Thorpe need not follow. Before proving our main result (Theorem \ref{thm:CP2example} below), let us demonstrate this now.

\begin{lemma}
\label{lemma:sharedHT}
Let $(M,g)$ be a compact oriented $4$-manifold whose curvature operator has the form
$$
\co=\begin{bmatrix}A&B\\B&C\end{bmatrix}\commas \left\{\begin{array}{lcl}
A&=&\operatorname{diag}(a_1,a_2,a_3),\\
B&=&\operatorname{diag}(b_1,b_2,b_3),\\
C&=&\operatorname{diag}(c_1,c_2,c_3)\end{array}\right.
$$
in a local oriented $\ipr{\,}{}$-orthonormal basis \eqref{eqn:basis}.  Then
\beqa
\label{eqn:abc}
2\chi(M)\pm3\tau(M)=\frac{1}{2\pi^2}\int_M\sum_{i=1}^3(a_i\pm b_i)(c_i\pm b_i)\,\dVg.
\eeqa
\end{lemma}

\begin{proof}
The Euler formula is the Pfaffian computation already used in the proof of Corollary \ref{cor:generalchi}; when written in the basis above, it is
\beqa
\label{eqn:Euler_int}
\chi(M)=\frac{1}{4\pi^2}\int_M\sum_{i=1}^3(a_ic_i+b_i^2)\,\dVg.
\eeqa
The signature formula also appeared in the proof of Theorem \ref{thm:relativeHT}:
$$
\tau(M)=\frac{1}{12\pi^2}\int_M\big(|W^+_{\scalebox{0.6}{$g$}}|^2 - |W^-_{\scalebox{0.6}{$g$}}|^2\big)\dVg.
$$
In the self-dual/anti-self-dual eigenbasis determined by the $\hsr$-splitting \mbox{$\Lambda^2=\Ls{g}{+}\oplus\Ls{g}{-}$,} the diagonal entries of the two Weyl blocks differ by the terms $b_i(a_i+c_i)$, and the algebraic Bianchi identity gives $b_1+b_2+b_3=0$.  A direct substitution gives
$$
\tau(M)=\frac{1}{6\pi^2}\int_M\sum_i b_i(a_i+c_i)\dVg.
$$
Combining the two formulas gives the displayed expressions for $2\chi\pm3\tau$.  
\end{proof}

In the shared-frame setting of Lemma \ref{rem:sharedframe}, $\Rmr(\Ls{h}{+},\Ls{h}{-})=0$ forces \mbox{$a_ic_i\ge0$} for each $i$, and therefore the Euler integrand \eqref{eqn:Euler_int} is nonnegative.  But \eqref{eqn:abc} shows that the Hitchin--Thorpe integrands need not be nonnegative.  For instance, the algebraic data
$$
\left\{\begin{array}{lcl}
(a_1,c_1,b_1)\!\!&=&\!\!(4,1,2),\\
(a_2,c_2,b_2)\!\!&=&\!\!(-2,-1/2,-1),\\
(a_3,c_3,b_3)\!\!&=&\!\!(-2,-1/2,-1)\end{array}\right.
$$
satisfy $b_1+b_2+b_3=0$ and $a_ic_i>0$ for each $i$, yet
$$
\sum_i(a_ic_i+b_i^2)=12
\comma
\sum_i b_i(a_i+c_i)=15.
$$
Thus the shared-frame condition is compatible, already pointwise, with one Hitchin--Thorpe density having the wrong sign.  We now realize this phenomenon globally, by constructing $\Rmr(\Ls{h}{+},\Ls{h}{-})=0$ on a manifold that admits no Einstein metric.  We will need the following function, the proof of whose existence is long and technical, but needed for our main result (Theorem \ref{thm:CP2example}).

\begin{prop}
\label{prop:profilesmoothing}
Set $s=9/10$.  There is a smooth function \mbox{$p\colon[0,s]\lra[0,\infty)$} with the following properties.  For $t \in[0,s]$, if
$$
a(t)\defeq\sin s-\int_t^s p(u)\,du
\comma
G(t)\defeq\frac{\sin^3t}{a(t)^3\cos t},
$$
then $0\leq p(t)\leq\cos t$, $p(t)=\cos t$ near $s$, $p(t)=\lambda t+O(t^3)$ near $0$ for some $\lambda>0$,
and there is a unique $r\in(0.28,0.30)$ such that
$$
p'(t)>0\text{ and }p(t)>G(t)\text{ for }0<t<r,
$$
while
$$
p'(t)<0\text{ and }p(t)<G(t)\text{ for }r<t<s.
$$
Moreover, $p(r)=G(r)$, $p'(r)=0$,
and the quotient $(p-G)/p'$ extends smoothly and positively across $t=r$.
\end{prop}

\begin{proof}
We first construct a piecewise smooth model and then show that it can be smoothed without changing the sign information.  Set $s=9/10$ and define
\beqa
\label{eqn:F}
F(x)\defeq \cos x-\frac{s-x}{2}\sin x-
\frac{1}{\cos x\left(1+\frac{(s-x)^2}{6}\right)^{\!3}}\cdot
\eeqa
Because $F(0.28)>0$ and $F(0.30)<0$, there exists $r\in(0.28,0.30)$ satisfying
\beqa
\label{eqn:profile_r_equation}
\underbrace{\,\cos r-\frac{s-r}{2}\sin r\,}_{\text{``$P$"}}=
\frac{1}{\cos r\left(1+\frac{(s-r)^2}{6}\right)^{\!3}}\cdot
\eeqa
On the right-hand interval $[r,s]$, define
$$
p_{\scalebox{0.4}{$R$}}(t)\defeq
\cos t-\frac{1}{2}\frac{(s-t)^2}{s-r}\sin r.
$$
Then $p_{\scalebox{0.4}{$R$}}(r)=P, p'_{\scalebox{0.4}{$R$}}(r)=0, p_{\scalebox{0.4}{$R$}}(t)\leq\cos t$, and $ p'_{\scalebox{0.4}{$R$}}(t)<0$ for all $t\in (r,s]$. If $a(t)=\sin s-\int_t^s p_{\scalebox{0.4}{$R$}}(u)\,du$ on $[r,s]$, then
$
a(r)=\sin r\big(1+\frac{(s-r)^2}{6}\big) > 0.
$
Thus \eqref{eqn:profile_r_equation} is precisely the equality $P=G(r)$. On the left-hand interval $[0,r]$, define
$$
p_{\scalebox{0.4}{$L$}}(t)\defeq P\phi_0(t/r)\comma
\phi_0(x)\defeq 2x-x^2.
$$
Then for $x=t/r$ and $a(t)=\sin s-\int_t^s p_{\scalebox{0.4}{$L$}}(u)\,du$ on $[0,r]$, one has
$$
a(rx)=a(r)\big(1-\kappa I(x)\big)
\commas
\kappa\defeq\frac{Pr}{a(r)} < \frac{9}{10}
\commas
I(x)=\frac23-x^2+\frac{x^3}{3}\cdot
$$
To prove that $p_{\scalebox{0.4}{$L$}}>G$ on $(0,r)$, it is enough to prove, using $P=G(r)$ and the fact that $\frac{\sin^3(rx)}{(rx)^3\cos(rx)}$ is increasing for $x \in (0,1)$, that
$$
(2x-x^2)\left(1-\frac9{10}\left(\frac23-x^2+\frac{x^3}{3}\right)\right)^{\!3}>x^3
\comma 0<x<1,
$$
which is verified by inspection.
Therefore $p_{\scalebox{0.4}{$L$}}>G$ on $(0,r)$.  Since $p'_{\scalebox{0.4}{$L$}}>0$ on $(0,r)$, the desired sign pattern holds on the left.  On the right, $p_{\scalebox{0.4}{$R$}}'<0$.  Since $p_{\scalebox{0.4}{$R$}}\leq\cos t$, we have $a\geq\sin t$ and $a'\leq \cos t$, and hence
\beqa
\label{eqn:G'}
\frac{G'}G=3\cot t-3\frac{a'}a+\tan t\geq\tan t>0\comma r\leq t \leq s.
\eeqa
Thus $G$ is increasing while $p_{\scalebox{0.4}{$R$}}$ is decreasing.  Since $p_{\scalebox{0.4}{$R$}}(r)=G(r)$, we obtain $p_{\scalebox{0.4}{$R$}}<G$ on $(r,s)$. To finish the proof, it remains to replace the piecewise model by a smooth profile. Let $p^{\scalebox{0.6}{$\mathrm{mod}$}}$ denote the piecewise model just constructed, and set
$$
a^{\scalebox{0.6}{$\mathrm{mod}$}}(t)\defeq\sin s-\int_t^s p^{\scalebox{0.6}{$\mathrm{mod}$}}(u)\,du .
$$
Since $a^{\scalebox{0.6}{$\mathrm{mod}$}}(0)>0$, choose $0<\varepsilon<a^{\scalebox{0.6}{$\mathrm{mod}$}}(0)/2$. After taking all subsequent smoothing parameters sufficiently small, every profile considered below remains in the set
$$
\mathcal P_\varepsilon \defeq \left\{p\in C^\infty([0,s]) : p\ge0,\ \int_0^s p(u)\,du<\sin s-\varepsilon\right\}\cdot
$$
What is more, the map
$$
p\in\mathcal P_\varepsilon \longmapsto a_p(t)\defeq\sin s-\int_t^s p(u)\,du
$$
is continuous in the \(C^0\)-topology; indeed, for $p,q \in P_{\scalebox{0.6}{$\vep$}}$, 
$$
|a_p(t)-a_q(t)| \leq \int_t^s|p(u)-q(u)|\,du \leq \|p-q\|_{\scalebox{0.6}{$C^0$}}(s-t) \leq s\|p-q\|_{\scalebox{0.6}{$C^0$}}, 
$$
so that $\|a_p-a_q\|_{\scalebox{0.6}{$C^0$}} \leq s\|p-q\|_{\scalebox{0.6}{$C^0$}}$. Furthermore, $a_p$ stays bounded away from zero on $P_{\scalebox{0.6}{$\vep$}}$ because $a_p(t)\geq a_p(0) > \vep$. Because of this, and because $\cos t\geq\cos s>0$ on $[0,s]$, the corresponding function
$$
p\in P_{\scalebox{0.6}{$\vep$}} \mapsto G_p(t)\defeq\frac{\sin^3t}{a_p(t)^3\cos t}
$$
is also continuous in the $C^0$-topology.  We show the estimate, since it will be used below.  If $p,q\in P_{\scalebox{0.6}{$\vep$}}$, then $a_p(t),a_q(t)\geq\vep$ for all $t\in[0,s]$.  Applying the mean value theorem to the function $x\mapsto x^{-3}$ at the points $a_p(t),a_q(t)$ gives
$$
\big|a_p(t)^{-3}-a_q(t)^{-3}\big|
\leq 3\vep^{-4}|a_p(t)-a_q(t)|.
$$
Therefore
\begin{align*}
|G_p(t)-G_q(t)|
&=\frac{\sin^3t}{\cos t}\,\big|a_p(t)^{-3}-a_q(t)^{-3}\big|\\
&\leq \frac{\sin^3s}{\cos s}\,3\vep^{-4}|a_p(t)-a_q(t)|\\
&\leq \frac{3s\sin^3s}{\vep^4\cos s}\,
\|p-q\|_{\scalebox{0.6}{$C^0$}},
\end{align*}
which implies that
\begin{equation}
\label{eqn:G-continuity-explicit}
\|G_p-G_q\|_{\scalebox{0.6}{$C^0$}}
\leq
\frac{3s\sin^3s}{\vep^4\cos s}\,
\|p-q\|_{\scalebox{0.6}{$C^0$}}.
\end{equation}
holds.  The same estimate applies to the piecewise model profiles used below, since the proof of the estimate only uses the lower bound $a\geq\vep$ and not smoothness.  In particular, inequalities involving $p-G_p$ are controlled by $C^0$-smallness of $p$ together with \eqref{eqn:G-continuity-explicit}, whereas inequalities involving $p'$ require a separate $C^0$-estimate for $p'$. We now smooth at the three points $0,r$, and $s$.  Although we can quote a standard mollification procedure (see, e.g., \cite[Appendix~C.5,~Theorem~7]{evans}), we give a full proof here as our argument requires preserving the sign pattern at the matching point. We first smooth at $0$.  Recall that
$$
p_{\scalebox{0.4}{$L$}}(t)=P\phi_0(t/r)\comma\phi_0(x)= 2x-x^2.
$$
Now choose a smooth cutoff $\beta\colon[0,\infty)\lra[0,1]$ such that
$$
\beta(y)=0\quad\text{for}\quad 0\leq y\leq \frac{1}{2}\comma
\beta(y)=1\quad\text{for}\quad y\geq 1,
$$
and set
$
B_{\scalebox{0.6}{$\beta$}}\defeq \|\beta'\|_{\scalebox{0.6}{$C^0([0,\infty))$}}.
$
For $0<\eta<\frac{1}{2}$, define
$$
p_{\scalebox{0.6}{$\eta$}}(t)
\defeq P\phi_{\scalebox{0.6}{$\eta$}}(t/r)\comma \phi_{\scalebox{0.6}{$\eta$}}(x)
\defeq 2x-\beta(x/\eta)x^2.
$$
Then
$$
\phi_{\scalebox{0.6}{$\eta$}}(x)=2x
\quad\text{for}\quad0\leq x\leq \frac{\eta}{2}\comma
\phi_{\scalebox{0.6}{$\eta$}}(x)=2x-x^2
\quad\text{for}\quad x\geq \eta.
$$
Thus $p_{\scalebox{0.6}{$\eta$}}(t)=(2P/r)t$ for $0\leq t\leq r\eta/2$, which is of the form $p_{\scalebox{0.6}{$\eta$}}(t)=\lambda t+O(t^3)$ with $\lambda=2P/r>0$, as desired.  Also, since $0\leq\beta\leq1$, we have
$$
\phi_{\scalebox{0.6}{$\eta$}}(x)=2x-\beta(x/\eta)x^2
\geq 2x-x^2=\phi_0(x).
$$
In particular, $p_{\scalebox{0.6}{$\eta$}}\geq0$ on $[0,s]$.  At those points where $p_{\scalebox{0.6}{$\eta$}}\neq p_{\scalebox{0.4}{$L$}}$, namely \mbox{$0< t<r\eta$,} we have
$
p_{\scalebox{0.6}{$\eta$}}(t) = P(2(t/r)-\beta(t/(r\eta))(t/r)^2)
\leq 2P\eta.
$
After decreasing $\eta$ if necessary, \mbox{$2P\eta<\cos(r\eta)$;} since $\cos t\geq\cos(r\eta)$ on $[0,r\eta]$, this gives $p_{\scalebox{0.6}{$\eta$}}(t)\leq\cos t$ on $[0,r\eta]$.  For $t\geq r\eta$, we have $p_{\scalebox{0.6}{$\eta$}}(t)=p_{\scalebox{0.4}{$L$}}(t)\leq\cos t$, because $p_{\scalebox{0.4}{$L$}}$ is increasing on $[r\eta,r]$ and $0< \phi_0(t/r) \leq 1$ on $[r\eta,r]$, so that $p(t) \leq p_{\scalebox{0.4}{$L$}}(r)\leq P = \cos r-\frac{s-r}{2}\sin r\leq \cos r \leq \cos t$.
We now check the two inequalities
$
p_{\scalebox{0.6}{$\eta$}}>G_{\text{$p_{\scalebox{0.6}{$\eta$}}$}}$ and $p'_{\scalebox{0.6}{$\eta$}}>0$
on the whole left-hand interval $(0,r)$.  We do this in three separate regions. First, suppose $0<t\leq r\eta/2$.  Then
$$
p_{\scalebox{0.6}{$\eta$}}(t)=\frac{2P}{r}t\comma
p'_{\scalebox{0.6}{$\eta$}}(t)=\frac{2P}{r}>0.
$$
Let $a_0\defeq a_{\text{$p_{\scalebox{0.4}{$L$}}$}}(0)>0$.  Since
$
|p_{\scalebox{0.6}{$\eta$}}(t)-p_{\scalebox{0.4}{$L$}}(t)|
\leq P\eta^2$ for $0\leq t\leq r\eta
$
and the two profiles agree for $t\geq r\eta$, we have
$
|a_{\text{$p_{\scalebox{0.6}{$\eta$}}$}}(0)-a_{\text{$p_{\scalebox{0.4}{$L$}}$}}(0)|
\leq sP\eta^2.
$
Taking $\eta$ so small that $sP\eta^2<a_0/2$, we get
$
a_{\text{$p_{\scalebox{0.6}{$\eta$}}$}}(0)\geq a_0/2>0.
$
For $0<t\leq r\eta/2$, therefore,
$$
a_{\text{$p_{\scalebox{0.6}{$\eta$}}$}}(t)
=
a_{\text{$p_{\scalebox{0.6}{$\eta$}}$}}(0)+\int_0^t p_{\scalebox{0.6}{$\eta$}}(u)\,du
=
a_{\text{$p_{\scalebox{0.6}{$\eta$}}$}}(0)+\frac{P}{r}t^2.
$$
Thus
$$
G_{\text{$p_{\scalebox{0.6}{$\eta$}}$}}(t)
=
\frac{\sin^3t}{a_{\text{$p_{\scalebox{0.6}{$\eta$}}$}}(t)^3\cos t}
\leq
\frac{8}{a_0^3\cos s}t^3.
$$
After decreasing $\eta$ so that
$
\frac{8}{a_0^3\cos s}(r\eta/2)^2<\frac{2P}{r},
$
we obtain
$$
G_{\text{$p_{\scalebox{0.6}{$\eta$}}$}}(t)<\frac{2P}{r}t=p_{\scalebox{0.6}{$\eta$}}(t)
\quad\text{for }0<t\leq r\eta/2.
$$
This proves both desired inequalities on $(0,r\eta/2]$. Second, suppose \mbox{$r\eta\leq t<r$.}  Since $p_{\scalebox{0.6}{$\eta$}}=p_{\scalebox{0.4}{$L$}}$ on $[r\eta,r)$, and since the modification of $p_{\scalebox{0.4}{$L$}}$ is supported in $[0,r\eta]$, the integral in $a_p(t) = \sin s -\int_t^s p(u)du$ does not see the modification when $t\geq r\eta$.  Hence
$$
a_{\text{$p_{\scalebox{0.6}{$\eta$}}$}}(t)=a_{\text{$p_{\scalebox{0.4}{$L$}}$}}(t)\comma
G_{\text{$p_{\scalebox{0.6}{$\eta$}}$}}(t)=G_{\text{$p_{\scalebox{0.4}{$L$}}$}}(t)\comma
r\eta\leq t<r.
$$
Thus $p_{\scalebox{0.6}{$\eta$}}>G_{\text{$p_{\scalebox{0.6}{$\eta$}}$}}$ and $p'_{\scalebox{0.6}{$\eta$}}>0$ on $[r\eta,r)$, because these are the already-verified inequalities for the model function $p_{\scalebox{0.4}{$L$}}$. Third, consider the transition region
$
K_{\scalebox{0.6}{$\eta$}}\defeq [r\eta/2,r\eta].
$
On this region, write $x=t/r\in[\eta/2,\eta]$.  Since $a_{\text{$p_{\scalebox{0.4}{$L$}}$}}(0)>0$ and $f(t) \defeq \frac{P}{2r}t$ has $f'(0) = \frac{P}{2r} > G_{\text{$p_{\scalebox{0.4}{$L$}}$}}(0) = 0$, there is a number $\eta_0>0$ such that
$$
G_{\text{$p_{\scalebox{0.4}{$L$}}$}}(rx)\leq \frac{P}{2}x
\quad\text{for}\quad0<x\leq\eta_0,
$$
by the Mean Value Theorem on $G_{\text{$p_{\scalebox{0.4}{$L$}}$}} - f$. Now take $0<\eta<\eta_0$; because
$$
p_{\scalebox{0.4}{$L$}}(rx)=P(2x-x^2)\geq Px
\quad\text{for}\quad0<x<1,
$$
we get
\begin{equation}
\label{eqn:left-margin-p-minus-G}
p_{\scalebox{0.4}{$L$}}(rx)-G_{\text{$p_{\scalebox{0.4}{$L$}}$}}(rx)
\geq \frac{P}{2}x
\geq \frac{P}{4}\eta
\quad\text{for}\quad rx \in K_{\scalebox{0.6}{$\eta$}}.
\end{equation}
Likewise,
\begin{equation}
\label{eqn:left-margin-derivative}
p'_{\scalebox{0.4}{$L$}}(t)
=
\frac{2P}{r}\left(1-\frac{t}{r}\right)
\geq \frac{P}{r}
\quad\text{on}\quad K_{\scalebox{0.6}{$\eta$}},
\end{equation}
after imposing $\eta<1/2$. We now estimate how much the cutoff changes the profile.  Since
$
\phi_{\scalebox{0.6}{$\eta$}}(x)-\phi_0(x)
=
\big(1-\beta(x/\eta)\big)x^2,
$
we have
\begin{equation}
\label{eqn:left-C0-p-change}
\|p_{\scalebox{0.6}{$\eta$}}-p_{\scalebox{0.4}{$L$}}\|_{\scalebox{0.6}{$C^0(K_{\scalebox{0.6}{$\eta$}})$}}\leq P\eta^2.
\end{equation}
Also,
$
\phi'_{\scalebox{0.6}{$\eta$}}(x)-\phi_0'(x)
=
2\big(1-\beta(x/\eta)\big)x-\beta'(x/\eta)\frac{x^2}{\eta}.
$
Since $x=\frac{t}{r}\leq\eta$ on $K_{\scalebox{0.6}{$\eta$}}$, this gives
$
|\phi'_{\scalebox{0.6}{$\eta$}}(x)-\phi_0'(x)|
\leq (2+B_{\scalebox{0.6}{$\beta$}})\eta.
$
Therefore
\begin{equation}
\label{eqn:left-C0-derivative-change}
\|p'_{\scalebox{0.6}{$\eta$}}-p'_{\scalebox{0.4}{$L$}}\|_{\scalebox{0.6}{$C^0(K_{\scalebox{0.6}{$\eta$}})$}}
\leq \frac{P}{r}(2+B_{\scalebox{0.6}{$\beta$}})\eta.
\end{equation}
Finally, applying \eqref{eqn:G-continuity-explicit} with $p=p_{\scalebox{0.6}{$\eta$}}$ and $q=p_{\scalebox{0.4}{$L$}}$, and using \eqref{eqn:left-C0-p-change}, gives
\begin{equation}
\label{eqn:left-C0-G-change}
\|G_{\text{$p_{\scalebox{0.6}{$\eta$}}$}}-G_{\text{$p_{\scalebox{0.4}{$L$}}$}}\|_{\scalebox{0.6}{$C^0(K_{\scalebox{0.6}{$\eta$}})$}}
\leq
\frac{3sP\sin^3s}{\vep^4\cos s}\eta^2.
\end{equation}
Combining \eqref{eqn:left-margin-p-minus-G}, \eqref{eqn:left-C0-p-change}, and \eqref{eqn:left-C0-G-change}, we obtain, on $K_{\scalebox{0.6}{$\eta$}}$,
\begin{align*}
p_{\scalebox{0.6}{$\eta$}}-G_{\text{$p_{\scalebox{0.6}{$\eta$}}$}}
&=
\big(p_{\scalebox{0.4}{$L$}}-G_{\text{$p_{\scalebox{0.4}{$L$}}$}}\big)
+
\big(p_{\scalebox{0.6}{$\eta$}}-p_{\scalebox{0.4}{$L$}}\big)
-
\big(G_{\text{$p_{\scalebox{0.6}{$\eta$}}$}}-G_{\text{$p_{\scalebox{0.4}{$L$}}$}}\big)\\
&\geq
\frac{P}{4}\eta
-
P\eta^2
-
\frac{3sP\sin^3s}{\vep^4\cos s}\eta^2.
\end{align*}
This yields $p_{\scalebox{0.6}{$\eta$}}>G_{\text{$p_{\scalebox{0.6}{$\eta$}}$}}$ on $K_{\scalebox{0.6}{$\eta$}}$ for
$
\eta<\frac{1}{4}\big(1+\frac{3s\sin^3s}{\vep^4\cos s}\big)^{\!-1}.
$ Similarly, combining \eqref{eqn:left-margin-derivative} and \eqref{eqn:left-C0-derivative-change}, we get
$$
p'_{\scalebox{0.6}{$\eta$}}
\geq
\frac{P}{r}-\frac{P}{r}(2+B_{\scalebox{0.6}{$\beta$}})\eta.
$$
After requiring $(2+B_{\scalebox{0.6}{$\beta$}})\eta<1$, this gives $p'_{\scalebox{0.6}{$\eta$}}>0$ on $K_{\scalebox{0.6}{$\eta$}}$.  Thus we have now proved $p_{\scalebox{0.6}{$\eta$}}>G_{\text{$p_{\scalebox{0.6}{$\eta$}}$}}$ and $p'_{\scalebox{0.6}{$\eta$}}>0$ on all three regions $(0,r\eta/2]$, $[r\eta/2,r\eta]$, and $[r\eta,r)$.  Hence the left-hand sign conditions hold on the entire interval $(0,r)$.  We now modify the right-hand function $p_{\scalebox{0.4}{$R$}}$ in a similar manner. Choose a smooth cutoff \mbox{$\gamma\colon[0,\infty)\lra[0,1]$} such that
$$
\gamma(y)=0\quad\text{for}\quad 0\leq y\leq \frac12
\comma
\gamma(y)=1\quad\text{for}\quad y\geq1,
$$
and set $B_{\scalebox{0.6}{$\gamma$}}\defeq\|\gamma'\|_{\scalebox{0.6}{$C^0([0,\infty))$}}$.  For $0<\nu<s-r$, define
$$
q_{\scalebox{0.6}{$\nu$}}(t)
\defeq
\gamma\!\left(\frac{s-t}{\nu}\right)\frac{(s-t)^2}{s-r}
\commas
p_{\scalebox{0.6}{$\nu$}}(t)
\defeq
\cos t-\frac12\sin r\,q_{\scalebox{0.6}{$\nu$}}(t)
\commas r\leq t\leq s.
$$
Thus $p_{\scalebox{0.6}{$\nu$}}=p_{\scalebox{0.4}{$R$}}$ on $[r,s-\nu]$, while $p_{\scalebox{0.6}{$\nu$}}=\cos t$ on $[s-\nu/2,s]$.  Moreover, since $0\leq q_{\scalebox{0.6}{$\nu$}}\leq (s-t)^2/(s-r)$, we have $p_{\scalebox{0.4}{$R$}}\leq p_{\scalebox{0.6}{$\nu$}}\leq \cos t.$
The perturbation is $C^0$-small: it is supported in $[s-\nu,s]$, wherein $(s-t)^2/(s-r)\leq \nu^2/(s-r)$.  Hence
\beqa
\label{eqn:C0_*}
\|p_{\scalebox{0.6}{$\nu$}}-p_{\scalebox{0.4}{$R$}}\|_{\scalebox{0.6}{$C^0([r,s])$}}
\leq
\frac{\sin r}{2(s-r)}\nu^2.
\eeqa
We now retune the matching point, treating $r$ as a parameter. For each $r\in I\defeq[0.28,0.30]$, write
$$
P(r)\defeq \cos r-\frac{s-r}{2}\sin r,
$$
and define $p_{\scalebox{0.4}{$R$}}^{\,r}$ and $p_{\scalebox{0.6}{$\nu$}}^{\,r}$ by the same formulae as above, with this value of $r$. Let
$$
a_{\scalebox{0.6}{$\nu$}}^{\,r}(t)\defeq \sin s-\int_t^s p_{\scalebox{0.6}{$\nu$}}^{\,r}(u)\,du
\commas
G_{\scalebox{0.6}{$\nu$}}^{\,r}(t)\defeq
\frac{\sin^3t}{(a_{\scalebox{0.6}{$\nu$}}^{\,r}(t))^3\cos t},
$$
and set
$$
\Psi_{\scalebox{0.6}{$\nu$}}(r)\defeq P(r)-G_{\scalebox{0.6}{$\nu$}}^{\,r}(r).
$$
For the unsmoothed right profile $p_{\scalebox{0.4}{$R$}}^{\,r}$, the corresponding function $\Psi_0$ is precisely the function $F$ given by \eqref{eqn:F}, and has opposite signs at $0.28$ and $0.30$. Moreover, the estimate \eqref{eqn:C0_*} is uniform for $r\in I$, so
$$
\big\|p_{\scalebox{0.6}{$\nu$}}^{\,r}-p_{\scalebox{0.4}{$R$}}^{\,r}\big\|_{\scalebox{0.6}{$C^0([r,s])$}}
\leq C\nu^2
$$
for a constant $C$ independent of $r$. Hence
$$
\big|a_{\scalebox{0.6}{$\nu$}}^{\,r}(r)-a_{\scalebox{0.4}{$R$}}^{\,r}(r)\big|
\leq C\nu^3,
$$
and therefore $\Psi_{\scalebox{0.6}{$\nu$}}\to\Psi_0$ uniformly on $I$. After decreasing $\nu$, the signs of $\Psi_{\scalebox{0.6}{$\nu$}}$ at $0.28$ and $0.30$ are still opposite. By the Intermediate Value Theorem, there is an $r_{\scalebox{0.6}{$\nu$}}\in(0.28,0.30)$ such that $\Psi_{\scalebox{0.6}{$\nu$}}(r_{\scalebox{0.6}{$\nu$}})=0$. Rename this value $r$. The preceding left-hand estimates are uniform for $r\in I$, so after decreasing $\eta$ if necessary, the left-hand construction still gives $p_{\scalebox{0.6}{$\eta$}}>G_{\text{$p_{\scalebox{0.6}{$\eta$}}$}}$ and $p'_{\scalebox{0.6}{$\eta$}}>0$ on $(0,r)$ for this retuned value of $r$. Since the left-hand smoothing is supported near $0$, it does not affect the value of $a(r)$, and therefore the piecewise function
$$
p(t)\defeq
\begin{cases}
p_{\scalebox{0.6}{$\eta$}}(t) & 0\leq t<r,\\[2pt]
p_{\scalebox{0.6}{$\nu$}}(t) & r\leq t\leq s
\end{cases}
$$
satisfies $p(r)=G_p(r)$.

It remains to check the right-hand signs for the retuned $r$. First, $p'_{\scalebox{0.6}{$\nu$}}<0$ on $(r,s-\nu]$ because $p_{\scalebox{0.6}{$\nu$}}=p_{\scalebox{0.4}{$R$}}$ there, and $p'_{\scalebox{0.6}{$\nu$}}=-\sin t<0$ on $[s-\nu/2,s]$. On the transition region $K_{\scalebox{0.6}{$\nu$}}\defeq[s-\nu,s-\nu/2]$, write $y=(s-t)/\nu\in[\frac12,1]$. Then
$$
-q'_{\scalebox{0.6}{$\nu$}}(t)
=
\frac{\nu}{s-r}\big(2y\gamma(y)+y^2\gamma'(y)\big)
\leq
\frac{(2+B_{\scalebox{0.6}{$\gamma$}})\nu}{s-r}.
$$
Therefore, on $K_{\scalebox{0.6}{$\nu$}}$,
$$
p'_{\scalebox{0.6}{$\nu$}}(t)
= -\sin t-\frac12\sin r\,q'_{\scalebox{0.6}{$\nu$}}(t)
\leq
-\sin(s-\nu)+\frac{\sin r}{2(s-r)}(2+B_{\scalebox{0.6}{$\gamma$}})\nu,
$$
which is negative after decreasing $\nu$ if necessary. Hence $p'_{\scalebox{0.6}{$\nu$}}<0$ on $(r,s]$. Finally, for $a=a_p$ on the right-hand interval, the inequalities $0\leq p_{\scalebox{0.6}{$\nu$}}\leq\cos t$ imply $a(t)\geq\sin t$ and $a'(t)=p_{\scalebox{0.6}{$\nu$}}(t)\leq\cos t$. Thus
$$
\frac{G'_p}{G_p}
=3\cot t-3\frac{a'}a+\tan t
\geq\tan t>0
\commas r\leq t\leq s.
$$
So $G_p$ is increasing on $[r,s]$, while $p_{\scalebox{0.6}{$\nu$}}$ is decreasing there. Since $p_{\scalebox{0.6}{$\nu$}}(r)=G_p(r)$, it follows that $p_{\scalebox{0.6}{$\nu$}}<G_p$ on $(r,s]$. Thus the retuned piecewise profile satisfies the required signs on both sides of $r$.

Finally, we smooth near $r$; indeed, although $p_{\scalebox{0.6}{$\eta$}}(r)=p_{\scalebox{0.6}{$\nu$}}(r)$, the piecewise function $p$ need not be smooth at $r$. Let us rename $p$ to $p^{\scalebox{0.5}{temp}}$ and put $G^{\scalebox{0.5}{temp}}\defeq G_{\text{$p^{\scalebox{0.5}{temp}}$}}$ (note that $G^{\scalebox{0.5}{temp}}$ is continuous). By the retuning of $r$, we have
$
p^{\scalebox{0.5}{temp}}(r)=G^{\scalebox{0.5}{temp}}(r).
$
We now remove a small neighborhood of $r$ and insert a cap there.  We shall choose $\lambda_1>0$ below.  Pick a ``height" value $H$ near $G^{\scalebox{0.5}{temp}}(r)$ and set
$$
C_{\scalebox{0.5}{$H$}}(t)\defeq H-\frac{\lambda_1}{2}(t-r)^2.
$$
For $\delta>0$ small, choose nondecreasing smooth cutoffs $\chi_-$ and $\chi_+$ with values in $[0,1]$ such that $\chi_-\big|_{[0,r-\delta]}=0$ and $\chi_-\big|_{[r-\delta/3,s]}=1$, while $\chi_+\big|_{[0,r+\delta/3]}=0$ and $\chi_+\big|_{[r+\delta,s]}=1$.  Then define $p_{\scalebox{0.5}{$H$}}\colon[0,s]\lra [0,\infty)$ by
$$
p_{\scalebox{0.5}{$H$}}(t)\defeq
\begin{cases}
p_{\scalebox{0.6}{$\eta$}}(t) & 0\leq t\leq r-\delta,\\[2pt]
(1-\chi_-(t))p_{\scalebox{0.6}{$\eta$}}(t)+\chi_-(t)C_{\scalebox{0.5}{$H$}}(t) & r-\delta\leq t\leq r-\delta/3,\\[2pt]
C_{\scalebox{0.5}{$H$}}(t) & |t-r|\leq \delta/3,\\[2pt]
(1-\chi_+(t))C_{\scalebox{0.5}{$H$}}(t)+\chi_+(t)p_{\scalebox{0.6}{$\nu$}}(t) & r+\delta/3\leq t\leq r+\delta,\\[2pt]
p_{\scalebox{0.6}{$\nu$}}(t) & r+\delta\leq t\leq s.
\end{cases}
$$
The cutoffs are constant near the endpoints of the transition intervals, so $p_{\scalebox{0.5}{$H$}}$ is smooth. The only remaining choice is the height $H$. Write
$$
P\defeq p^{\scalebox{0.5}{temp}}(r)=G^{\scalebox{0.5}{temp}}(r).
$$
Since $p_{\scalebox{0.6}{$\eta$}}=p_{\scalebox{0.4}{$L$}}$ near $r$ and $p_{\scalebox{0.6}{$\nu$}}=p_{\scalebox{0.4}{$R$}}$ near $r$, there are positive constants $\alpha_-$ and $\alpha_+$ such that, for $0<y\ll1$,
$$
p_{\scalebox{0.6}{$\eta$}}(r-y)=P-\alpha_-y^2+O(y^3)
\quad\text{and}\quad
p_{\scalebox{0.6}{$\nu$}}(r+y)=P-\alpha_+y^2+O(y^3).
$$
Choose $0<\lambda_1<2\min\{\alpha_-,\alpha_+\}$. We claim that the height equation $H=G_{\text{$p_{\scalebox{0.5}{$H$}}$}}(r)$ has a solution with $|H-P|=O(\delta^3)$. Indeed, if $|H-P|\leq M\delta^3$, then $p_{\scalebox{0.5}{$H$}}-p^{\scalebox{0.5}{temp}}$ is supported in an interval of length $O(\delta)$ and is $O(\delta^2)$ there, uniformly for bounded $M$. Hence
$$
G_{\text{$p_{\scalebox{0.5}{$H$}}$}}(r)-G^{\scalebox{0.5}{temp}}(r)=O(\delta^3).
$$
Choosing $M$ larger than the implicit constant, the continuous function $\Phi(H)\defeq H-G_{\text{$p_{\scalebox{0.5}{$H$}}$}}(r)$ has opposite signs at $P-M\delta^3$ and $P+M\delta^3$. Therefore there is a height $H_0$ with
$$
H_0=G_{\text{$p_{\scalebox{0.5}{$H_0$}}$}}(r)
\quad\text{and}\quad
|H_0-P|=O(\delta^3).
$$
Let us verify that the capped profile still satisfies the bound $0\leq p\leq\cos t$.  Since $P=G^{\scalebox{0.5}{temp}}(r)>0$ and $P<\cos r$ (indeed, $P=\cos r-\frac{s-r}{2}\sin r$), the gap between $P$ and both $0$ and $\cos r$ is positive.  As $H_0=P+O(\delta^3)$ and $\cos t=\cos r+O(\delta)$ for $|t-r|\leq\delta$, after decreasing $\delta$ we have
$$
0<C_{\scalebox{0.5}{$H_0$}}(t)<\cos t
\commas |t-r|\leq\delta.
$$
The already constructed functions $p_{\scalebox{0.6}{$\eta$}}$ and $p_{\scalebox{0.6}{$\nu$}}$ satisfy $0\leq p_{\scalebox{0.6}{$\eta$}},p_{\scalebox{0.6}{$\nu$}}\leq\cos t$ on their respective domains.  Since the transition pieces defining $p_{\scalebox{0.5}{$H_0$}}$ are convex combinations of $C_{\scalebox{0.5}{$H_0$}}$ with $p_{\scalebox{0.6}{$\eta$}}$ or $p_{\scalebox{0.6}{$\nu$}}$, the final capped profile also satisfies $0\leq p_{\scalebox{0.5}{$H_0$}}(t)\leq\cos t$ on $[0,s]$.
Rename $p_{\scalebox{0.5}{$H_0$}}$ as $p$. Then, on $|t-r|\leq\delta/3$,
$$
p(t)=G_p(r)-\frac{\lambda_1}{2}(t-r)^2.
$$
In particular $p'(t)=-\lambda_1(t-r)$ on $[r-\delta/3,r+\delta/3]$. Also, after decreasing $\eta,\nu,\delta$ if necessary, we can ensure that $G_p'(r)>0$, since $G_p$ can be made $C^1$-close to $G_{\text{$p_{\scalebox{0.6}{$\nu$}}$}}$ near $r$ and $G_{\text{$p_{\scalebox{0.6}{$\nu$}}$}}'(r)>0$ by the computation above. Therefore $(p-G_p)'(r)=-G_p'(r)<0$, and, after decreasing $\delta$ once more, we obtain
$$
p>G_p\commas p'>0\quad\text{on }[r-\delta/3,r)
\commas
p<G_p\commas p'<0\quad\text{on }(r,r+\delta/3].
$$
It remains only to check the two transition intervals
$$
I_-\defeq[r-\delta,r-\delta/3]
\comma
I_+\defeq[r+\delta/3,r+\delta].
$$
On $I_-$, the choice of $\lambda_1$ gives $C_{\scalebox{0.5}{$H_0$}}-p_{\scalebox{0.6}{$\eta$}}>0$ for sufficiently small $\delta$; also $p'_{\scalebox{0.6}{$\eta$}}>0$ and $C'_{\scalebox{0.5}{$H_0$}}>0$ there. Since $\chi'_-\geq0$, the identity
$$
p'=(1-\chi_-)p'_{\scalebox{0.6}{$\eta$}}+\chi_-C'_{\scalebox{0.5}{$H_0$}}+\chi'_-(C_{\scalebox{0.5}{$H_0$}}-p_{\scalebox{0.6}{$\eta$}})
$$
shows that $p'>0$ on $I_-$. Similarly, on $I_+$ one has $p_{\scalebox{0.6}{$\nu$}}-C_{\scalebox{0.5}{$H_0$}}<0$, $p'_{\scalebox{0.6}{$\nu$}}<0$, and $C'_{\scalebox{0.5}{$H_0$}}<0$; since $\chi'_+\geq0$, the identity
$$
p'=(1-\chi_+)C'_{\scalebox{0.5}{$H_0$}}+\chi_+p'_{\scalebox{0.6}{$\nu$}}+\chi'_+(p_{\scalebox{0.6}{$\nu$}}-C_{\scalebox{0.5}{$H_0$}})
$$
shows that $p'<0$ on $I_+$. Finally, the margins $p^{\scalebox{0.5}{temp}}-G^{\scalebox{0.5}{temp}}$ on $I_-$ and \mbox{$G^{\scalebox{0.5}{temp}}-p^{\scalebox{0.5}{temp}}$} on $I_+$ are bounded below by $c\,\delta$ for some $c>0$, while \mbox{$p-p^{\scalebox{0.5}{temp}}=O(\delta^2)$} and $G_p-G^{\scalebox{0.5}{temp}}=O(\delta^3)$ there. Hence the signs of $p-G_p$ on the two transition intervals are preserved: $p-G_p>0$ on $I_-$ and $p-G_p<0$ on $I_+$. Thus the final smooth profile satisfies all the stated sign conditions. Moreover,
$$
\frac{p-G_p}{p'}=
\frac{G_p(r)-G_p(t)-\frac{\lambda_1}{2}(t-r)^2}{-\lambda_1(t-r)}
$$
extends smoothly and positively at $t=r$, with limiting value $G_p'(r)/\lambda_1$.  (By Hadamard's Lemma, $p-G_p=(t-r)h$ for some smooth function $h$, so that $\frac{p-G_p}{p'} = -\frac{h}{\lambda_1}$.) This completes the proof.
\end{proof}

\begin{prop}
\label{lemma:cp2core}
There is a smooth Riemannian pair $(g_{\scalebox{0.4}{$C$}},h_{\scalebox{0.4}{$C$}})$ satisfying \emph{$\text{Rm}_{\scalebox{0.6}{$g_{\scalebox{0.4}{$C$}}$}}(\Ls{h_{\scalebox{0.4}{$C$}}}{+},\Ls{h_{\scalebox{0.4}{$C$}}}{-})=0$} on a manifold $C\cong\CP^2\setminus B^4$ such that, near $\partial C$,
$$
g_{\scalebox{0.4}{$C$}}=h_{\scalebox{0.4}{$C$}}=dt^2+\sin^2t\,g_{\scalebox{0.5}{$\Sph^3$}}
$$
for $t$ near $s=9/10$.
\end{prop}

\begin{proof}
Let $p$ be as in Proposition \ref{prop:profilesmoothing}, with
$s=\frac{9}{10}$ and $a(t)=\sin s-\int_t^s p(u)\,du$. Note that $0 \leq p(t) \leq \cos t$, hence $a(t) \geq \sin t$ and $|a'|=|p|\leq1$; note also that $a_0>0$, because $p(t)=\lambda t+O(t^3)$ near $0$ while $\cos t=1+O(t^2)$, so that $\int_0^s p(u)du < \sin s$. Now, let $\sigma_1,\sigma_2,\sigma_3$ be the standard left-invariant coframe on $\Sph^3$, normalized by $d\sigma_1=-2\sigma_2\wedge\sigma_3$ and cyclically.  Set
$$
g=dt^2+a(t)^2(\sigma_1^2+\sigma_2^2)+\sin^2t\,\sigma_3^2\comma 0\leq t\leq s.
$$
With $e^0=dt$, $e^1=a\sigma_1$, $e^2=a\sigma_2$, and $e^3=\sin t\,\sigma_3$, the curvature operator $\co$ takes the form of Lemma \ref{rem:sharedframe}, with \mbox{$a_1=a_2=\frac{a''}{a}$,} $a_3=-1$ and
$$
c_1=c_2=-\frac{\sin^2t}{a^4}+\frac{a'\cos t}{a\sin t}\comma
c_3=-\frac4{a^2}+\frac{3\sin^2t}{a^4}+\left(\frac{a'}a\right)^{\!2}\cdot
$$
Set
$$
G(t)\defeq\frac{\sin^3t}{a(t)^3\cos t}
\commas
D(t)\defeq a'(t)\cos t-\frac{\sin^3t}{a(t)^3}=\cos t\,(p(t)-G(t)).
$$
We now make repeated use of Proposition \ref{prop:profilesmoothing}. Since $a_1=a_2=p'/a$ and $c_1=c_2=D/(a\sin t)$, we have that $a_1c_1=a_2c_2>0$ on $(0,r)\cup(r,s]$ and $a_1=a_2=c_1=c_2=0$ at $t=r$.  Also,
$$c_3=\frac1{a^2}\bigg(\!\!-\!4+3\Big(\frac{\sin t}{a}\Big)^{\!2}+(a')^2\bigg)<0$$ on $(0,s]$, hence $a_3c_3>0$
and so \mbox{Lemma \ref{rem:sharedframe}} applies pointwise on $(0,s]$. Near $0$ we have that $p(t)=\lambda t+O(t^3)$, so that $$a(t)=a_0+\int_0^t p(u)du = a_0+\frac{\lambda}{2}t^2+O(t^4)$$
with $a_0>0$ and $\lambda>0$.  Since $\sin t=t+O(t^3)$, the metric extends smoothly across the $\CP^1$-bolt (i.e., the collapsed $\mathbb{S}^1$-fiber leaves a smooth $\mathbb{CP}^1$ fixed surface), and the resulting manifold is $\CP^2\setminus B^4$.  Topologically, the collapsing $\sigma_3$-circle is the Hopf fiber of the $\Sph^3$-orbits, so the bolt is a copy of $\CP^1$ whose normal disk bundle has Euler number $+1$ with the chosen orientation; this disk bundle is $\CP^2\setminus B^4$.  Near $s$ we have that $p(t)=\cos t$, hence $a(t)=\sin t$; the metric is therefore exactly round near the boundary. It remains to define $h_{\scalebox{0.4}{$C$}}$.  On the open set where $a_2c_2=a_1c_1>0$, set
$$
x_0\defeq\sqrt{\frac{c_1c_3}{a_1a_3}}\comma x_1=1=x_2\comma
x_3\defeq\sqrt{\frac{a_1c_3}{c_1a_3}},
$$
where $a_1=a_2,c_1=c_2$, and $a_3=-1$, and define
$$
h_{\scalebox{0.4}{$C$}}
\defeq
x_0^{-1}(e^0)^2+(e^1)^2+(e^2)^2+x_3^{-1}(e^3)^2.
$$
The equations in Lemma \ref{rem:sharedframe} are then satisfied (after relabeling), so that $\Rmr(\Ls{h_{\scalebox{0.4}{$C$}}}{+},\Ls{h_{\scalebox{0.4}{$C$}}}{-})=0$.  At $r$, recall that $a_1=c_1=0$. Since $$a_1=p'/a \comma c_1=\frac{\cos t\,(p-G)}{a\sin t},$$
the quotient $c_1/a_1$ is a positive smooth multiple of \mbox{$(p-G)/p'$,} and therefore extends smoothly and positively across $r$. Hence $x_0$ and $x_3$ extend smoothly and positively there.  At the bolt $t=0$, we have $a_1(0)=c_1(0)=\lambda/a_0$ (the latter via L'H\^opital's Rule), while $a_3(0)=-1$ and \mbox{$c_3(0)=-4/a_0^2$,} so $$x_0^{-1}(e^0)^2+x_3^{-1}(e^3)^2=\frac{a_0}{2}(dt^2+\sin^2t\,\sigma_3^2).$$ Thus $h_{\scalebox{0.4}{$C$}}$ has the same smooth polar behavior as $g_{\scalebox{0.4}{$C$}}$.  On the round boundary collar at $s$, $a_i=c_i=-1$ for all $i=1,2,3$, hence $x_0=x_3=1$ and $h_{\scalebox{0.4}{$C$}}=g_{\scalebox{0.4}{$C$}}$.
\end{proof}

\begin{prop}
\label{prop:gluing}
Suppose compact oriented $4$-manifolds with boundary carry Riemannian pairs $(g_i,h_i)$ satisfying \emph{$\text{Rm}_{\scalebox{0.6}{$g_i$}}(\Ls{h_i}{+},\Ls{h_i}{-})=0$}, and suppose that $g_i=h_i=dt^2+\sin^2t\,g_{\scalebox{0.5}{$\Sph^3$}}$ on matching boundary collars.  If the collars are identified by orientation-reversing collar maps which identify the round metrics on the $\Sph^3$-slices and match the collar coordinate in the usual way, then the resulting glued manifold carries a smooth Riemannian pair $(g,h)$ satisfying \emph{$\Rmr(\Ls{h}{+},\Ls{h}{-})=0$}.
\end{prop}

\begin{proof}
The orientation-reversing boundary identification is the usual one for producing an oriented glued manifold. The compatibility of the collar identifications means that the $g_i$ and $h_i$ have identical pullbacks on the overlap; hence they glue to smooth Riemannian metrics $g$ and $h$. Away from the seam, the condition is one of the original pointwise conditions. On the collar itself, the metric is round and $h=g$, so the condition holds there as well: the round metric is Einstein, and therefore its curvature preserves the ordinary Hodge splitting. Since $\Rmr(\Ls{h}{+},\Ls{h}{-})=0$ is a local pointwise condition, the glued pair satisfies it globally.
\end{proof}

\begin{thm}[A Hitchin--Thorpe-violating example]
\label{thm:CP2example}
$\#_5\CP^2$ admits a Riemannian pair $(g,h)$ satisfying \emph{$\Rmr(\Ls{h}{+},\Ls{h}{-})=0$}, but no Einstein metric.
\end{thm}

\begin{proof}
In the unit round sphere $\Sph^4\subset\RR^5$, choose the five vertices of a regular $4$-simplex centered at the origin. Their pairwise inner products are $-1/4$, so their mutual spherical distances are $\arccos(-1/4)$. Because $s = 9/10$, we have that
$2s=9/5<\arccos(-1/4)$, so that the five open geodesic balls of radius $s$ centered at these points are pairwise disjoint. Remove these balls from $\Sph^4$ and glue into the resulting five boundary components five copies of $C\cong\CP^2\setminus B^4$ from Proposition \ref{lemma:cp2core}. The boundary collars are round and agree with $$g_{\scalebox{0.4}{$C$}}=h_{\scalebox{0.4}{$C$}}=dt^2+\sin^2t\,g_{\scalebox{0.5}{$\Sph^3$}},$$ so Proposition \ref{prop:gluing} gives a smooth Riemannian pair $(g,h)$ on the glued manifold satisfying $\Rmr(\Ls{h}{+},\Ls{h}{-})=0$. Topologically, each inserted copy of $C$ performs connected sum with $\CP^2$, so the glued manifold is $\#_5\CP^2$. Therefore, because
$
\chi(\#_5\CP^2)=7$ and
$\tau(\#_5\CP^2)=5$,
we have
$$
2\chi=14<15=3|\tau|.
$$
(In particular, $\rhoz(g,h) \geq \frac{15}{14}$.) Thus $\#_5\CP^2$ violates the Hitchin--Thorpe inequality, and therefore cannot admit an Einstein metric.
\end{proof}

The nonexistence of Einstein metrics on 4-manifolds is a rich subject, with obstructions coming both from Hitchin--Thorpe and Seiberg--Witten theory \cite{lebrun3,ishida}. Our construction goes in a different direction: Theorem \ref{thm:CP2example} exhibits a manifold violating Hitchin--Thorpe, hence admitting no Einstein metric, but satisfying the weaker relative condition \eqref{eqn:relativeHT}.

\section*{References}
\renewcommand*{\bibfont}{\footnotesize}
\printbibliography[heading=none]
\end{document}